\documentclass[a4paper,11pt]{amsart}
\usepackage[left=2.7cm,right=2.7cm,top=3.5cm,bottom=3cm]{geometry}

\usepackage{amsthm,amssymb,amsmath,amsfonts,mathrsfs,amscd}
\usepackage[latin1]{inputenc}
\usepackage[all]{xy}
\usepackage{latexsym}
\usepackage{longtable}

\newfont{\cyr}{wncyr10 scaled 1100}
\newfont{\cyrr}{wncyr9 scaled 1000}

\theoremstyle{plain}
\newtheorem{theorem}{Theorem}[section]
\newtheorem{proposition}[theorem]{Proposition}
\newtheorem{lemma}[theorem]{Lemma}
\newtheorem{corollary}[theorem]{Corollary}

\theoremstyle{definition}
\newtheorem{conjecture}[theorem]{Conjecture}
\newtheorem{definition}[theorem]{Definition}
\newtheorem{assumption}[theorem]{Assumption}
\newtheorem{question}[theorem]{Question}

\theoremstyle{remark}
\newtheorem{remark}[theorem]{Remark}

\newcommand{\Q}{\mathbb Q}

\newcommand{\Z}{\mathbb Z}

\newcommand{\C}{\mathbb C}

\DeclareMathOperator{\Pic}{Pic}
\DeclareMathOperator{\End}{End}
\DeclareMathOperator{\Aut}{Aut}
\DeclareMathOperator{\Frob}{Frob}

\DeclareMathOperator{\Hom}{Hom}

\DeclareMathOperator{\Gal}{Gal}
\DeclareMathOperator{\GL}{GL}

\DeclareMathOperator{\BK}{\Lambda}
\newcommand{\BBK}{X}
\DeclareMathOperator{\CH}{CH}

\newcommand{\res}{\mathrm{res}}
\newcommand{\cores}{\mathrm{cores}}

\newcommand{\tr}{\mathrm{tr}}
\newcommand{\ord}{\mathrm{ord}}

\newcommand{\cont}{\mathrm{cont}}

\newcommand{\Sha}{\mbox{\cyr{X}}}
\newcommand{\Shaa}{\mbox{\cyrr{X}}}

\usepackage[usenames]{color}
\definecolor{Indigo}{rgb}{0.2,0.1,0.7}
\definecolor{Violet}{rgb}{0.5,0.1,0.7}
\definecolor{White}{rgb}{1,1,1}
\definecolor{Green}{rgb}{0.1,0.9,0.2}

\newcommand{\longmono}{\mbox{$\lhook\joinrel\longrightarrow$}}

\newcommand{\smallmat}[4]{\bigl(\begin{smallmatrix}#1&#2\\#3&#4\end{smallmatrix}\bigr)}

\newcommand{\invlim}{\mathop{\varprojlim}\limits}

\newcommand{\T}{\mathbb T}

\newcommand{\E}{\mathcal E}

\newcommand{\F}{\mathbb{F}}

\newfont{\gotip}{eufb10 at 12pt}

\newcommand{\cO}{{\mathcal O}}

\newcommand{\p}{\mathfrak{p}}

\DeclareMathOperator{\Reg}{Reg}

\include{thebibliography}

\begin{document}

\title[A refined Beilinson--Bloch conjecture]{A refined Beilinson--Bloch conjecture\\for motives of modular forms}
\date{}
\author{Matteo Longo and Stefano Vigni}
\thanks{}

\begin{abstract}
We propose a refined version of the Beilinson--Bloch conjecture for the motive associated with a modular form of even weight. This conjecture relates the dimension of the image of the relevant $p$-adic Abel--Jacobi map to certain combinations of Heegner cycles on Kuga--Sato varieties. We prove theorems in the direction of the conjecture and, in doing so, obtain higher weight analogues of results for elliptic curves due to Darmon.
\end{abstract}

\address{Dipartimento di Matematica, Universit\`a di Padova, Via Trieste 63, 35121 Padova, Italy}
\email{mlongo@math.unipd.it}
\address{Dipartimento di Matematica, Universit\`a di Genova, Via Dodecaneso 35, 16146 Genova, Italy}
\email{vigni@dima.unige.it}



\maketitle


\section{Introduction} 

Let $N\geq3$ be an integer, let $k\geq4$ be an even integer and let $f\in S_k^{\rm new}(\Gamma_0(N))$ be a normalized newform of weight $k$ and level $\Gamma_0(N)$, whose $q$-expansion will be denoted by
\[ f(q)=\sum_{n\geq1}a_nq^n. \]  
Let $p\nmid N$ be a prime number and let $\p\,|\,p$ be a prime ideal of the ring of integers $\cO_F$ of the totally real field $F$ generated by the Fourier coefficients $a_n$ of $f$. Finally, let $K$ be a number field. To these data we may attach a $p$-adic Abel--Jacobi map 
\[ {\rm AJ}_K:{\CH}^{k/2}\bigl(\tilde\E_N^{k-2}/{K}\bigr)_0\otimes F_\p\longrightarrow H^1_f(K,V_\p) \]
where $F_\p$ is the completion of $F$ at $\p$, $\tilde\E_N^{k-2}$ is the Kuga--Sato variety of level $N$ and weight $k$, $V_\p$ is a twist of the $\p$-adic representation attached to $f$ and $H^1_f(K,V_\p)$ is its Bloch--Kato Selmer group over $K$ (here the subscript ``$f$'' stands for ``finite'' and should not be confused with the modular form $f$). The Beilinson--Bloch conjectures (\cite{Beil}, \cite{Bloch}) connect the values of the $L$-functions of algebraic varieties over number fields to global arithmetic properties of these varieties (see, e.g., \cite{Schneider} for an introduction). In particular, they state that the $F_\p$-dimension of the image $X_\p(K)$ of ${\rm AJ}_K$ is equal to the order of vanishing of the complex $L$-function $L(f\otimes K,s)$ of $f$ over $K$ at its center of symmetry $s=k/2$. Moreover, if $\tilde\rho_\p$ denotes this dimension then the leading term of the derivative of order $\tilde\rho_\p$ of $L(f\otimes K,s)$ at $s=k/2$ is predicted up to multiplication by elements of $\Q^\times$. When $K$ is an imaginary quadratic field of discriminant coprime to $Np$ or $K=\Q$, important results towards this conjecture (at least in low rank situations) have been obtained by combining Nekov\'a\v{r}'s generalization of Kolyvagin's theory to Chow groups of Kuga--Sato varieties (\cite{Nek}) with Zhang's formula of Gross--Zagier type for higher weight modular forms (\cite{Zh}). More recently, the Beilinson--Bloch conjectures have been subsumed within the Tamagawa number conjecture of Bloch and Kato (\cite{BK}), which predicts (by using Fontaine's theory of $p$-adic representations) the value of the non-zero rational factor that was not made explicit in the original conjectures. 

The goal of the present article is to investigate refined -- or equivariant -- analogues of these conjectures in which, roughly speaking, $L$-functions are replaced by Heegner cycles. 

To better explain our work, let us recall that refined versions of the Birch and Swinnerton-Dyer conjecture (BSD conjecture, for short) for a rational elliptic curve $E$ were first proposed by Mazur and Tate in \cite{MT}. In that article, the role of $L$-functions was played by certain combinations of modular symbols with coefficients in the group algebra $\Q[\Gal(\Q(\zeta_M)/\Q)]$, called ``theta elements'' and denoted by $\theta_{E,M}$; here $M\geq1$ is an integer and $\zeta_M$ is a primitive $M$-th root of unity. The Mazur--Tate refined conjecture of BSD type states that $\theta_{E,M}$ belongs to a power $r$ of the augmentation ideal $I$ of $\Q[\Gal(\Q(\zeta_M)/\Q)]$ that can be predicted in terms of the rank of the Mordell--Weil group $E(\Q)$ and the number of primes of split multiplicative reduction for $E$ dividing $M$. This conjecture describes also the leading value of $\theta_{E,M}$, which is defined as the image of $\theta_{E,M}$ in the quotient $I^r/I^{r+1}$. Extensions and analogues of this conjecture for Artin $L$-functions and for $L$-functions of more general motives have also been formulated, and partial results have been proved (see, e.g., \cite{Bu}, \cite{BF}, \cite{DFG}, \cite{Gr}, \cite{Ru} and the references therein). 
 
Moving from \cite{MT} and the observation that modular symbols and Heegner points enjoy similar formal properties, Darmon proposed in \cite{Dar} refined versions \emph{\`a la} Mazur--Tate of the BSD conjecture, where modular symbols are replaced by Heegner points. Later on, Bertolini and Darmon began the systematic study of $p$-adic analogues of the BSD conjecture in which the relevant $p$-adic $L$-functions are defined in terms of distributions of Heegner (and Gross--Heegner) points on Shimura curves attached either to definite or to indefinite quaternion algebras (see \cite{BD96}, \cite{BD98}, \cite{BD99}, \cite{BD01}, \cite{BD-Iwasawa}).

Our aim in this paper is to formulate and study refined versions of the Beilinson--Bloch conjecture for the motive associated with the modular form $f$; in this context, the role of the Heegner points appearing in \cite{Dar} is played by higher-dimensional Heegner cycles in the sense of Nekov\'a\v{r} (\cite{Nek}). We hope that our work, offering an equivariant refinement of the above mentioned conjectures in which the complex $L$-function of a modular form is replaced by an algebraically defined one, can be viewed as complementary to the results of Burns and of Burns--Flach on Stark's conjectures and Tamagawa numbers of motives (see, e.g., \cite{Bu}, \cite{BF}).  

In order to state our main results more precisely, we need some notation. Let $K$ be an imaginary quadratic field of discriminant coprime to $Np$ in which all the primes dividing $N$ split, let $T$ be a square-free product of primes that are inert in $K$ and do not divide $Np$ and let $K_T$ be the ring class field of $K$ of conductor $T$. Write $\cO_\p$ for the completion of $\cO_F$ at $\p$. As recalled in \S \ref{galois-sec} and \S \ref{secAJ}, there is a natural way to introduce an $\cO_\p$-lattice $A_\p$ inside $V_\p$, and to all these data we may attach a Heegner cycle $y_{T,\p}\in\BK_\p(K_T)\subset H^1_{\rm cont}(K_T,A_\p)$ where $\BK_\p(K_T)$ is the image of the $\cO_\p$-integral version of the Abel--Jacobi map ${\rm AJ}_{K_T}$ and $H^1_{\rm cont}$ denotes continuous cohomology (see \S \ref{selmer-subsec} and \S \ref{secHC}). Set $G_T:=\Gal(K_T/K_1)$ and $\Gamma_T:=\Gal(K_T/K)$, consider the theta element
\[ \theta_{T,\p}:=\sum_{\sigma\in G_T}\sigma(y_{T,\p})\otimes\sigma\in\BK_\p(K_T)\otimes_{\cO_\p}\!\cO_\p[G_T] \] 
and let $\theta_{T,\p}^*$ be the image of $\theta_{T,\p}$ via the involution sending $\sigma\in G_T$ to $\sigma^{-1}$. Taking suitable trace-like operators to $K$ we obtain elements $\zeta_{T,\p}$ and $\zeta_{T,\p}^*$ that may be naturally viewed as belonging to $\BK_\p(K_S)\otimes_{\cO_\p}\!\cO_\p[\Gamma_S]$ whenever $T\,|\,S$. 

Now let $S$ be a square-free product of primes that are inert in $K$ and do not divide $Np$, then define the arithmetic $L$-function attached to $S$ and $\p$ as
\[ \mathcal L_{S,\p}:=\bigg(\sum_{T\mid S}a_T\zeta_{T,\p}\bigg)\!\otimes\!\bigg(\sum_{T\mid S}a_T^*\zeta_{T,\p}^*\bigg)\in\BK_\p(K_S)^{\otimes2}\otimes_{\cO_\p}\!\cO_\p[\Gamma_S], \]
where $a_T$ and $a^*_T$ are explicit elements of $\cO_\p[\Gamma_S]$ that are defined in \eqref{def-multipl} below in terms of the M\"obius function and the quadratic character of $K$. 

The finite-dimensional $F_\p$-vector space $X_\p(K)$ splits under the action of the non-trivial element of $\Gal(K/\Q)$ as a direct sum 
\[ \BBK_\p(K)=\BBK_\p(K)^+\oplus\BBK_\p(K)^- \] 
of its eigenspaces. Set $\rho^\pm_\p:=\dim_{F_\p}\bigl(\BBK_\p(K)^\pm\bigr)$ and
\[ \rho_\p:=\begin{cases}\max\{\rho^+_\p,\rho^-_\p\}-1&\text{if $\rho^+_\p\not=\rho^-_\p$},\\[3mm]\rho^+_\p=\rho^-_\p&\text{otherwise}.\end{cases} \]  
As a consequence of the Beilinson--Bloch conjecture, the function $\p\mapsto\rho_\p$ is expected to be constant and, since the order of vanishing of $L(f\otimes K,s)$ at $s=k/2$ is odd, the case $\rho_\p^+=\rho_\p^-$ should never occur. Let $I_{\Gamma_S}$ be the augmentation ideal of $\cO_\p[\Gamma_S]$ (to simplify our notation, we suppress dependence on $\p$). Finally, write $J(S)$ for the cokernel of the map 
\[ H^1_f(K,A_\p/pA_\p)\longrightarrow \bigoplus_{\ell\mid S}H^1_f(K_\lambda, A_\p/pA_\p) \]
where $K_\lambda$ is the completion of $K$ at the unique prime $\lambda$ above $\ell$. Our results apply to all prime numbers $p$ outside a finite set $\Sigma$ that we introduce in \S \ref{exc-set}; in fact, one crucial feature that we require of the prime $p$ is that the Galois representation attached to $A_\p$ be irreducible with non-solvable image.

\begin{theorem} \label{thm-intro}
Let $p$ be a prime number such that $p\not\in\Sigma$, let $S$ be a product of primes that are inert in $K$ and do not divide $Np$, and let $\p$ be a prime ideal of $\cO_F$ above $p$.
\begin{enumerate}
\item $\mathcal L_{S,\p}\in\BK_\p(K_S)^{\otimes2}\otimes_{\cO_\p}\!I^{2\rho_\p}_{\Gamma_S}$. 
\item Suppose that $p\,|\,\ell+1$ for all prime numbers $\ell\,|\,S$. If $|\rho_\p^+-\rho_\p^-|=1$ then the image $\tilde{\mathcal L}_{S,\p}^{(p)}$ of $\mathcal L_{S,\p}$ in 
\[ \bigl(\Lambda_\p(K_S)^{\otimes2}/p\Lambda_\p(K_S)^{\otimes2}\bigr)\otimes_{\cO_\p}\!\bigl(I_{\Gamma_S}^{2\rho_\p}/I_{\Gamma_S}^{2\rho_\p+1}\bigr) \] 
belongs to the natural image of 
\[ \bigl(\Lambda_\p(K)^{\otimes2}/p\Lambda_\p(K)^{\otimes2}\bigr)\otimes_{\cO_\p}\!\bigl(I_{\Gamma_S}^{2\rho_\p}/I_{\Gamma_S}^{2\rho_\p+1}\bigr). \]
\item Assume that $p\,|\,\ell+1$ for all prime numbers $\ell\,|\,S$ and that the $\p$-part $\Sha_\p(f/K)$ of the Shafarevich--Tate group of $f$ over $K$ is finite. If $|\rho_\p^+-\rho_\p^-|=1$ and $p$ divides $|\Sha_\p(f/K)|\cdot|J(S)|$ then $\tilde{\mathcal L}_{S,\p}^{(p)}=0$.  
\end{enumerate}
\end{theorem}

Theorem \ref{thm-intro}, which corresponds to Corollary \ref{coro4.14} in the main body of the text, provides a higher weight analogue of a theorem of Darmon for elliptic curves over $\Q$ (\cite{Dar}), and at the same time can be viewed as a partial result towards a \emph{refined Beilinson--Bloch conjecture} for modular forms. This perspective is addressed in the last part of the paper (Section \ref{sec-regulators}), where we relate $\mathcal L_{S,\p}$ to an axiomatic theory of regulators of Mazur--Tate type in this setting. The conjectural picture motivating Theorem \ref{thm-intro} is stated as Question \ref{ques}. We believe that these regulators can be explicitly defined using Nekov\'a\v{r}'s theory of $p$-adic height pairings (\cite{nek4}), and we plan to come back to these issues in a subsequent paper. 

We conclude by remarking that Theorem \ref{thm-intro} is a consequence of analogous results for the elements $\zeta_{S,\p}$ (Theorem \ref{theorem5.27}). It is worth pointing out that all these results are based on a congruence property enjoyed by Heegner cycles (Theorem \ref{div-thm}); namely, generalized Kolyvagin derivatives (called \emph{Darmon--Kolyvagin derivatives} in this article and studied in \S \ref{kolyvagin-subsec}) of Heegner cycles are zero modulo $p^m$ if their order is less than the $\cO_\p/p^m\cO_\p$-rank of $H^1_f(K,A_\p/p^mA_\p)$. As a by-product of Theorem \ref{div-thm}, if $\ell$ is a prime not dividing $N$, inert in $K$ and such that $p\,|\,\ell+1$ then in Theorem \ref{Heegner-theorem} we give a bound (in terms of $p$ and the dimension of $H^1_f(K,A_\p/pA_\p)$ over $\cO_\p/p\cO_\p$) on the $\cO_\p/p\cO_\p$-dimension of the Galois module generated by Heegner cycles inside $\BK_\p(K_\ell)/p\BK_\p(K_\ell)$.

\subsubsection*{Notation and conventions} 
Unless specified otherwise, unadorned tensor products $\otimes$ are taken over $\Z$. 

The cardinality of a (finite) set $X$ is denoted either by $\#X$ or by $|X|$, according to convenience. 

If $K$ is a field then set $G_K:=\Gal(\bar K/K)$, where $\bar K$ is a fixed algebraic closure of $K$.
For any continuous $G_K$-module $M$ let $H^i(K,M)$ denote the $i$-th cohomology group of $G_K$ with coefficients in $M$. If $K/F$ is a field extension then  
\[ \res_{K/F}:H^i(F,M)\longrightarrow H^i(K,M),\qquad\cores_{K/F}:H^i(K,M)\longrightarrow H^i(F,M) \] 
denote the restriction and corestriction maps in cohomology, respectively. Recall that for $K/F$ finite and Galois there is an equality 
\begin{equation} \label{res-cores-norm} 
\res_{K/F}\circ\cores_{K/F}={\bf N}_{K/F}
\end{equation} 
where ${\bf N}_{K/F}:=\sum_{\sigma\in\Gal(K/F)}\sigma$ is the Galois norm (or trace) operator acting on $H^i(K,M)$. 

Fix algebraic closures $\bar\Q$ of $\Q$ and $\bar\Q_\ell$ of $\Q_\ell$ for any prime number $\ell$, and then fix field embeddings $\bar\Q\hookrightarrow\bar\Q_\ell$ for every $\ell$. Let $\Q_\ell^{\rm nr}$ be the maximal unramified extension of $\Q_\ell$ inside $\bar\Q_\ell$ and write $F_\ell$ for the arithmetic Frobenius in $\Gal(\Q_\ell^{\rm nr}/\Q_\ell)$. With an abuse of notation, when dealing with a $G_\Q$-module that is unramified at $\ell$ we shall often adopt the same symbol to denote a lift of $F_\ell$ to $G_{\Q_\ell}$ (and its image in $G_\Q$). 

Finally, if $L/E$ is a Galois extension of number fields, $\lambda$ is a prime of $E$ that is unramified in $L$ and $\lambda'$ is a prime of $L$ above $\lambda$ then $\Frob_{\lambda'/\lambda}\in\Gal(L/E)$ denotes the Frobenius substitution at $\lambda'$; the conjugacy class of $\Frob_{\lambda'/\lambda}$ in $\Gal(L/E)$ will be denoted by $\Frob_\lambda$ (notation not reflecting dependence on $L$).

\subsubsection*{Acknowledgements} 

It is a pleasure to thank Jan Nekov\'a\v{r} for enlightening conversations on some of the topics of this paper.

\section{Beilinson--Bloch conjecture for modular forms} 

In this section $f\in S_k(\Gamma_0(N))$ is a normalized newform of (even) weight $k$ and level $\Gamma_0(N)$ and $p$ is a prime number such that $p\nmid2N(k-2)!\phi(N)$, where $\phi$ is Euler's function.

\begin{remark}
For the arguments developed in this section, a more natural choice of $p$ would simply require that $p\nmid2N$ and $p>k-1$, as explained in \cite[\S 6.5]{Nek2}. However, in this case the notation becomes more complicated and some neatly stated results, for instance \cite[Proposition 2.1]{Nek}, require substantial modifications to make them consistent. In order to emphasize the new aspects of our work without indulging in unenlightening technicalities, we therefore decided to work under the above simplifying assumption. 
\end{remark}

\subsection{Galois representations} \label{galois-sec}

Denote by $Y_N$ the affine modular curve over $\Q$ of level $\Gamma(N)$, whose complex points are given by $Y_N(\C)=\Gamma(N)\backslash\mathcal H$ where $\mathcal H$ is the upper half plane. Let $j:Y_N\hookrightarrow X_N$ be the proper smooth compactification of $Y_N$.

For any integer $n\geq1$ define the sheaves 
\[ \mathcal F_n:={\rm Sym}^{k-2}\bigl(R^1\pi_*(\Z/p^n\Z)\bigr)(k/2-1),\qquad\mathcal F:=\invlim_n\mathcal F_n\]
(both $\mathcal F$ and $\mathcal F_n$ depend on $p$, but we suppress this dependence to simplify notations).   

Let $B:=\Gamma(N)/\Gamma_0(N)$, consider the projector $\Pi_B:=(\#B)^{-1}\sum_{b\in B}b\in\Z_p[B]$ and define
 \[ J_p:=\Pi_B H^1_\text{\'et}(X_N\otimes\bar\Q,j_*\mathcal F)(k/2). \] 
Denote by $\T$ the Hecke algebra generated over $\Z$ by the standard Hecke operators $T_\ell$ for primes $\ell\nmid N$. Let $\theta_f:\T\rightarrow\cO_F$ be the morphism associated with $f$, where $F$ is the totally real number field generated over $\Q$ by the Fourier coefficients of $f$ and $\mathcal O_F$ is its ring of integers. The Hecke algebra $\T$ acts on $J_p$, as explained in \cite[pp. 101--102]{Nek}. Set $I_f:=\ker(\theta_f)$ and define 
\[ A_p:=\bigl\{x\in J_p\mid I_f\cdot x=0\bigr\}. \]
Then $A_p$, which should be regarded as a higher weight analogue of the Tate module of an abelian variety, is equipped with a continuous $\cO_F$-linear action of the absolute Galois group $G_\Q:=\Gal(\bar\Q/\Q)$ and is (isomorphic to) the $k/2$-twist of the representation attached to $f$ by Deligne (\cite{Del}). More precisely, $A_p$ is a free $\cO_F\otimes\Z_p$-module of rank $2$ such that for every prime $\ell\nmid Np$ the arithmetic Frobenius $F_\ell$ at $\ell$ acting on $A_p$ satisfies 
\begin{equation} \label{char-pol}
\det\bigl(1-F_\ell X\,|\,A_p\bigr)=1-\frac{a_\ell}{\ell^{\frac{k}{2}-1}}X+\ell X^2.
\end{equation}
As pointed out in \cite[p. 102]{Nek}, there is a map $J_p\rightarrow A_p$ that is both $\T$-equivariant and $G_\Q$-equivariant. 

\subsection{Kuga--Sato varieties}

In this subsection we briefly recall basic definitions and facts about Kuga--Sato varieties, along the lines of \cite{Del}, \cite[\S 2]{Nek}, \cite[\S 1]{Sch} (see also \cite[Appendix A]{BDP} by Conrad for a generalization to the relative situation).

Let $\pi:\E_N\rightarrow Y_N$ be the universal elliptic curve and $\bar\pi:\bar\E_N\rightarrow X_N$ the universal generalized elliptic curve, which is proper but not smooth. Define 
\[ \bar\pi_{k-2}:\bar\E_N^{k-2}\longrightarrow X_N \] 
to be the fiber product of $k-2$ copies of $\bar\E_N$ over $X_N$. If $k\geq4$ then $\bar\E_N^{k-2}$ is singular and we call its canonical desingularization $\tilde\E_N^{k-2}$ constructed by Deligne (\cite{Del}) the \emph{Kuga--Sato variety} of level $N$ and weight $k$. 

The level $N$ structure on $\bar\E_N$ induces a homomorphism $(\Z/N\Z)^2\times X_N\hookrightarrow\mathscr E_N$ of group schemes over $X_N$, where $\mathscr E_N$ is the N\'eron model of $\bar\E_N$ over $X_N$. Therefore $(\Z/N\Z)^2$ acts by translations on $\bar\E_N$. Moreover, $\Z/2\Z$ acts as multiplication by $-1$ in the fibers, and this gives an action of $(\Z/N\Z)^2 \rtimes(\Z/2\Z)$ on $\bar\E_N$. Finally, the symmetric group $S_{k-2}$ on $k-2$ letters acts on $\bar\E_N^{k-2}$ by permutation of the factors, and this gives an action of 
\[ \Gamma_{k-2}:=\big((\Z/N\Z)^2 \rtimes(\Z/2\Z)\big)^{k-2}\rtimes S_{k-2} \] 
on $\bar\E_N^{k-2}$ by automorphisms on the fibers of $\pi_{k-2}$, which extends canonically to an action of $\Gamma_{k-2}$ on $\tilde\E_N^{k-2}$.

Now define the homomorphism $\epsilon:\Gamma_{k-2}\rightarrow\{\pm1\}$ to be trivial on $(\Z/N\Z)^{2(k-2)}$, the product map on $(\Z/2\Z)^{k-2}$ and the sign character on $S_{k-2}$. Finally, let 
\[ \Pi_\epsilon\in\Z[1/2N(k-2)!][\Gamma_{k-2}] \] 
be the projector associated with $\epsilon$. 

Then, by \cite[Proposition 2.1]{Nek} (see also \cite[Theorem 1.2.1]{Sch} and \cite[II, Proposition 2.4]{Nek2} for the analogous result with coefficients in $\Q_p$), we have 
\[ H^1_{\text{\'et}}(X_N\otimes\bar \Q,j_*\mathcal F_n)(1)=\Pi_\epsilon H^{k-1}_{\text{\'et}}
\bigl(\tilde\E_N^{k-2}\otimes\bar\Q,\Z/p^n\Z\bigr)(k/2). \] 
Furthermore, thanks to \cite[Lemma 2.2]{Nek}, we know that $H^1_{\text{\'et}}(X_N,j_*\mathcal F)$ is torsion free and that $H^1_{\text{\'et}}(X_N,j_*\mathcal F/p^mj_*\mathcal F)$ is canonically isomorphic to $H^1_{\text{\'et}}(X_N,j_*\mathcal F)/p^mH^1_{\text{\'et}}(X_N,j_*\mathcal F)$, for all integers $m\geq 1$. Combining these facts we obtain a map 
\begin{equation} \label{scholl-2}
H^{k-1}_{\text{\'et}}\bigl(\tilde\E_N^{k-2}\otimes\bar\Q,\Z_p\bigr)(k/2)\longrightarrow J_p\longrightarrow A_p
\end{equation} 
that factors through $\Pi_\epsilon H^{k-1}_{\text{\'et}}\bigl(\tilde\E_N^{k-2}\otimes\bar\Q,\Z/p^n\Z\bigr)(k/2)$. 

\subsection{Abel--Jacobi maps} \label{secAJ} 

Fix a field ${L}$ of characteristic $0$, denote by $\bar {L}$ an algebraic closure of ${L}$ and let  
\begin{equation} \label{AJ}
\Phi_{p,L}:\CH^{k/2}\bigl(\tilde\E_N^{k-2}/{L}\bigr)_0\longrightarrow H^1_{\cont}\Big(\!{L},H^{k-1}_{\text{\'et}}\bigl(\tilde\E_N^{k-2}\otimes\bar{L},\Z_p(k/2)\bigr)\!\Big)
\end{equation}
be the $p$-adic Abel--Jacobi map (see \cite[\S 9]{Jan}). Here $\CH^{k/2}\bigl(\tilde\E_N^{k-2}/{L}\bigr)_0$ is the group of homologically trivial cycles of codimension $k/2$ on $\tilde\E_N^{k-2}$ defined over ${L}$ modulo rational equivalence and $H^1_{\rm cont}$ denotes continuous cohomology. Equivalently,
\begin{equation} \label{chow-eq}
\CH^{k/2}\bigl(\tilde\E_N^{k-2}/{L}\bigr)_0=\ker\Big(\!\CH^{k/2}\bigl(\tilde\E_N^{k-2}/{L}\bigr)\longrightarrow H^k_{\text{\'et}}\bigl(\tilde\E_N^{k-2}\otimes\bar {L},\Z_p(k/2)\bigr)\!\Big), 
\end{equation}  
where $\CH^{k/2}\bigl(\tilde\E_N^{k-2}/{L}\bigr)$ is the group of cycles of codimension $k/2$ on $\tilde\E_N^{k-2}$ defined over ${L}$ modulo rational equivalence. Indeed, using the Lefschetz principle and comparison isomorphisms between \'etale and singular cohomology over $\C$, it can be proved that the right hand side of \eqref{chow-eq} does not depend on $p$ (see, e.g., \cite[\S 1.3]{Nek3} for details).

Composing \eqref{scholl-2} and \eqref{AJ} and extending $\Z_p$-linearly, we get a map 
\begin{equation} \label{AJ3}
{\rm AJ}_{f,p,L}:{\CH}^{k/2}\bigl(\tilde\E_N^{k-2}/{L}\bigr)_0\otimes\Z_p\longrightarrow H^1_{\cont}({L},A_p).
\end{equation} 

Now we localize (or, rather, complete) the representation $A_p$ at a prime ideal $\p$ of $\mathcal O_F$ dividing $p$. More precisely, if $\p$ is such a prime then denote by $\cO_\p$ the completion of $\mathcal O_F$ at $\p$ and set $A_\p:=A_p\otimes_{\mathcal O_F\otimes\Z_p}\mathcal O_\p$, which is a free $\mathcal O_\p$-module of rank $2$ equipped with a $G_\Q$-action. It follows that $A_p=\prod_{\p\,|\,p}A_\p$, the product being taken over all prime ideals of $\cO_F$ above $p$. Fix once and for all a prime ideal $\p$ as above. Composing the map ${\rm AJ}_{f,p,L}$ introduced in \eqref{AJ3} with the one induced by the canonical projection $A_p\twoheadrightarrow A_\p$, we get an $\cO_\p$-linear map
\begin{equation} \label{AJ4}
{\rm AJ}_{f,\p,L}:{\CH}^{k/2}\bigl(\tilde\E_N^{k-2}/L\bigr)_0\otimes \mathcal O_\p\longrightarrow H^1_{\cont}(L,A_\p).
\end{equation}
If $L$ is a Galois extension of $L'$ then ${\rm AJ}_{f,\p,L}$ is $\Gal(L/L')$-equivariant with respect to the natural Galois actions on domain and codomain (\cite[Proposition 4.2]{Nek}). For simplicity, from here on we write ${\rm AJ}_L$ for ${\rm AJ}_{f,\p,L}$, understanding that we are fixing a prime $\p$ of $F$ above $p$.  

Finally, let us introduce another map that will be used in \S \ref{secHC}. Since the Abel--Jacobi map commutes with automorphisms of the underlying variety, the map ${\rm AJ}_{f,p,L}$ in \eqref{AJ3} factors through 
\[ \Pi_\epsilon\Big(\!\CH^{k/2}\bigl(\tilde\E_N^{k-2}/{L}\bigr)_0\otimes\Z_p\Big)=\Pi_\epsilon\Big(\!\CH^{k/2}\bigl(\tilde\E_N^{k-2}/{L}\bigr)\otimes\Z_p\Big); \] 
the equality follows from \cite[Prop. 2.1]{Nek}, see also \cite[p. 105]{Nek}. Thus \eqref{AJ4} yields a map 
\begin{equation} \label{AJ-map} 
\Psi_{f,\p,L}:\Pi_B\Pi_\epsilon\Big(\!\CH^{k/2}\bigl(\tilde\E_N^{k-2}/{L}\bigr)\otimes\mathcal O_\p\Big)\longrightarrow H^1_{\cont}({L},A_\p).
\end{equation}
This map is $\T$-equivariant and if $L$ is Galois over $\Q$ then it is $\Gal({L}/\Q)$-equivariant as well (use \cite[Proposition 4.2]{Nek} and apply the projection $A_p\twoheadrightarrow A_\p$, which is both $\T$- and $\Gal(L/\Q)$-equivariant).  

\subsection{Selmer groups} \label{selmer-subsec}

Let $E$ be a number field and denote by $G_E:=\Gal(\bar E/E)$ its absolute Galois group. Let $\mathcal V$ be a $p$-adic representation of $G_E$ unramified outside a finite set $\Xi$ of places of $E$ containing all the archimedean primes and the primes above $p$. If $v$ is a prime of $E$ above $p$ then, as in \cite[Sections 3 and 5]{BK}, define
\[ H^1_f(E_v,\mathcal V):=\ker\Big(H^1_{\rm cont}(E_v,\mathcal V)\longrightarrow H^1_{\rm cont}\bigl(E_v,\mathcal V\otimes_{\Q_p}\!B_{\rm cris}\bigr)\Big), \] 
where $B_{\rm cris}$ is Fontaine's crystalline ring of periods (see, e.g., \cite[Section 1]{BK}, and do not confuse the subscript ``$f$'' in $H^1_f$ with our fixed modular form $f$!). If $v$ is a prime of $E$ not dividing $p$ then write $I_v:=\Gal(\bar E_v/E_v^{\rm ur})$ for the inertia subgroup of $\Gal(\bar E_v/E_v)$, where $E_v^{\rm ur}$ denotes the maximal unramified extension of $E_v$. The unramified cohomology of $\mathcal V$ at $v$ is defined as
\[ H^1_{\rm ur}(E_v,\mathcal V):=H^1_{\cont}\bigl(\Gal(E_v^{\rm ur}/E_v),\mathcal V^{I_v}\bigr)\simeq\ker\Big(H^1_{\cont}(E_v,\mathcal V)\longrightarrow H^1_{\cont}(I_v,\mathcal V)\Big), \]
the isomorphism coming from the inflation-restriction exact sequence (i.e., the exact sequence of low degree terms in the relevant Hochschild--Serre spectral sequence). Finally, for such a prime $v$ of $E$ set
\[ H^1_f(E,\mathcal V):=H^1_{\rm ur}(E_v,\mathcal V). \]

\begin{definition} \label{BK-V-dfn}
The \emph{Bloch--Kato Selmer group} $H^1_f(E,\mathcal V)$ is the $\Q_p$-subspace of $H^1_{\rm cont}(E,\mathcal V)$ consisting of those classes whose localizations lie in $H^1_f(E_v,\mathcal V)$ for all primes $v$ of $E$. 
\end{definition}

Let $G_{E,\Xi}$ denote the Galois group over $E$ of the maximal extension of $E$ unramified outside $\Xi$; then $\mathcal V$ is a representation of $G_{E,\Xi}$ and $H^1_f(E,\mathcal V)$ is a subspace of the finite-dimensional $\Q_p$-vector space $H^1_{\rm cont}(G_{E,\Xi},\mathcal V)$, hence $H^1_f(E,\mathcal V)$ has finite dimension over $\Q_p$.

Now we specialize the previous discussion to the case where 
\begin{equation} \label{V-eq}
\mathcal V=H^{k-1}_{\text{\'et}}\bigl(\tilde\E_N^{k-2}\otimes\bar{E},\Q_p(k/2)\bigr).
\end{equation} 
It is well known that $\mathcal V$ is unramified outside the primes of $E$ dividing $Np$; in light of this, from here on we take 
\begin{equation} \label{sigma-eq}
\Xi:=\bigl\{\text{$v$ place of $E$}\;\;\big|\;\;\text{$v\,|\,Np$ or $v\,|\,\infty$}\bigr\}. 
\end{equation}

\begin{remark}
With $\mathcal V$ as in \eqref{V-eq}, the Selmer group $H^1_f(E,\mathcal V)$ of Definition \ref{BK-V-dfn} is equal to the one originally defined in \cite{BK} and later studied, e.g., by Besser in \cite{Bes}. In particular, it is smaller than the group considered by Nekov\'a\v{r} in \cite{Nek}; this is due to the fact that no local conditions at the places of $E$ dividing $N$ are imposed in \cite{Nek} (cf. \cite[p. 118]{Nek}).
\end{remark}

Let 
\begin{equation} \label{AJ-Q-p-eq}
\Phi_{p,E}\otimes\Q_p:\CH^{k/2}\bigl(\tilde\E_N^{k-2}/E\bigr)_0\otimes\Q_p\longrightarrow H^1_{\cont}\Big(\!E,H^{k-1}_{\text{\'et}}\bigl(\tilde\E_N^{k-2}\otimes\bar{E},\Q_p(k/2)\bigr)\!\Big)
\end{equation}
be the map induced by the Abel--Jacobi map in \eqref{AJ}.

\begin{theorem}[Nizio\l, Nekov\'a\v{r}, Saito] \label{niziol-thm}
There is an inclusion
\begin{equation} \label{AJ-inclusion-eq}
{\rm im}(\Phi_{p,E}\otimes\Q_p)\subset H^1_f\Big(\!E,H^{k-1}_{\text{\rm{\'et}}}\bigl(\tilde\E_N^{k-2}\otimes\bar{E},\Q_p(k/2)\bigr)\!\Big). 
\end{equation}
In particular, ${\rm im}(\Phi_{p,E}\otimes\Q_p)$ is a finite-dimensional vector space over $\Q_p$. 
\end{theorem}

\begin{proof} Let $v$ be a prime of $E$ and, for simplicity, set 
\[ \mathcal V_v:=H^{k-1}_{\text{\'et}}\bigl(\tilde\E_N^{k-2}\otimes\bar{E}_v,\Q_p(k/2)\bigr).\] We need to show that there is an inclusion
\[ {\rm im}(\Phi_{p,E_v}\otimes\Q_p)\subset H^1_f(E_v,\mathcal V_v), \]
where the map $\Phi_{p,E_v}\otimes\Q_p$ is defined as in \eqref{AJ-Q-p-eq} with $E$ replaced by $E_v$. If $v\nmid p$ then the weight-monodromy conjecture (\cite[p. 238]{Sai2}) is known to hold for compactified Kuga--Sato varieties over $E_v$ (\cite{Sai}, \cite{Sai2}), and so $H^1_{\rm cont}(E_v,\mathcal V_v)=0$ by \cite[Proposition 2.5]{Nek3}. On the other hand, if $v\,|\,p$ then $\tilde\E_N^{k-2}$ has good reduction at $v$ (recall that $\tilde\E_N^{k-2}$ has good reduction outside $N$ and $p\nmid N$), hence ${\rm im}(\Phi_{p,E_v}\otimes\Q_p)\subset H^1_f(E_v,\mathcal V_v)$ by \cite[Theorem 3.2]{Niz}. Finally, the last assertion follows from the finite dimensionality over $\Q_p$ of the right hand side of \eqref{AJ-inclusion-eq}. \end{proof}

\begin{remark} The result used above was proved in \cite{Niz} under a projectivity assumption on the relevant algebraic varieties, but this stronger condition can be dispensed with, as explained in \cite[Theorem 3.1]{Nek3}. \end{remark}

We will now consider Selmer groups of $A_\p$ and of quotients of it, and use Theorem \ref{niziol-thm} to describe them. For simplicity, assume that the prime number $p$ does not ramify in $F$. Define the $F_\p$-vector space $V_\p:=A_\p\otimes _{\cO_\p}\!F_\p$. For every integer $m\geq1$ define $W_\p:=A_\p\otimes\Q_p/\Z_p$, so that $W_\p[p^m]=A_\p/p^mA_\p$. For any place $v$ of $E$ there are maps 
\[ \varphi_v:H^1(E_v,A_\p)\longrightarrow H^1(E_v,V_\p),\qquad\pi_v:H^1(E_v,A_\p)\longrightarrow H^1(E_v,W_\p[p^m]) \] 
induced by the canonical arrows $A_\p\hookrightarrow V_\p$ and $A_\p\twoheadrightarrow W_\p[p^m]$. Set 
\[ H^1_f(E_v,A_\p):=\varphi_v^{-1}\bigl(H^1_f(E_v,V_\p)\bigr),\qquad H^1_f(E_v,W_\p[p^m]):=\pi_v\bigl(H^1_f(E_v,A_\p)\bigr). \]

In the following definition $M$ denotes either $A_\p$ or $W_\p[p^m]$.

\begin{definition} \label{BK-p^m-dfn} 
The \emph{Bloch--Kato Selmer group} $H^1_f({E},M)$ of $M$ over $E$ is the subgroup of $H^1_{\rm cont}({E},M)$ consisting of the classes whose localizations lie in $H^1_f({E}_v,M)$ for all $v$. 
\end{definition}

If $\Xi$ is as in \eqref{sigma-eq} then $A_\p$ is a $G_{E,\Xi}$-module and $H^1_f({E},W_\p[p^m])$ is a subgroup of the finite group $H^1(G_{E,\Xi},W_\p[p^m])$, hence $H^1_f({E},W_\p[p^m])$ is a finite $\mathcal O_\p/p^m\mathcal O_\p$-module. 

As in \eqref{V-eq}, set $\mathcal V:=H^{k-1}_{\text{\'et}}\bigl(\tilde\E_N^{k-2}\otimes\bar{E},\Q_p(k/2)\bigr)$. To clarify the various relations between Abel--Jacobi maps and Selmer groups, observe that there is a commutative diagram 
\begin{equation}\label{diagram-selmer}
\xymatrix@C=35pt{
{\CH}^{k/2}\bigl(\tilde\E_N^{k-2}/{E}\bigr)_0\otimes\Q_p \ar[d]\ar[rr]^-{\Phi_{p,E}\,\otimes\,\Q_p}
&&H^1_f(E,\mathcal V)\ar[d]^-\lambda\\
{\CH}^{k/2}\bigl(\tilde\E_N^{k-2}/{E}\bigr)_0\otimes F_\p \ar[r]^-{{\rm AJ}_{E}\,\otimes\,F_\p}
&H^1_{\cont}({E},V_\p)& H^1_f(E,V_\p)\ar@{_(->}[l]\\
{\CH}^{k/2}\bigl(\tilde\E_N^{k-2}/{E}\bigr)_0\otimes \mathcal O_\p \ar[r]^-{{\rm AJ}_{E}}
\ar[u]\ar@{->>}[d]&
H^1_{\cont}({E},A_\p)\ar[u]_-\varphi\ar[d]^-{\varpi}& 
H^1_f(E,A_\p)\ar@{_(->}[l]\ar[u]_-\varphi\ar@{->>}[d]^-{\varpi}\\
{\CH}^{k/2}\bigl(\tilde\E_N^{k-2}/E\bigr)_0\otimes(\mathcal O_\p/p^m\mathcal O_\p) 
\ar[r]^-{{\rm AJ}_{E, m}}
&H^1_{\cont}({E},W_\p[p^m])& H^1_f(E,W_\p[p^m])\ar@{_(->}[l]\
}
\end{equation} 
where
\begin{itemize}
\item the map $\lambda$ comes from the map $\mathcal V\rightarrow V_\p$ induced by the map $\mathcal V\rightarrow A_p$ in \eqref{scholl-2} by projecting from $A_p$ onto $A_\p$ and then composing with the inclusion $A_\p\hookrightarrow V_\p$;
\item the maps $\varphi$ and $\varpi$ are induced by $A_\p\hookrightarrow V_\p$ and $A_\p\twoheadrightarrow W_\p[p^m]$, respectively; 
\item the unlabeled vertical arrows are induced by the natural maps $\Q_p\hookrightarrow F_\p$, $\cO_\p\hookrightarrow F_\p$ and $\cO_\p\twoheadrightarrow\cO_\p/p^m\cO_\p$;
\item the maps ${\rm AJ}_E\otimes F_\p$ and ${\rm AJ}_{E,m}$ are induced by multiplication by elements of $F_\p$ and $\cO_\p/p^m\cO_\p$, respectively.
\end{itemize} 

\begin{corollary} \label{niziol-coro-2}
There are inclusions
\begin{enumerate} 
\item ${\rm im}({\rm AJ}_{E}\otimes F_\p)\subset H^1_f(E,V_\p)$;
\item ${\rm im}({\rm AJ}_{E})\subset H^1_f(E,A_\p)$; 
\item ${\rm im}({\rm AJ}_{E,m})\subset H^1_f(E,W_\p[p^m])$. 
\end{enumerate}
In particular, the $F_\p$-vector space ${\rm im}({\rm AJ}_{E}\otimes F_\p)$ has finite dimension.
\end{corollary}

\begin{proof} All the inclusions follow easily from the definitions and the commutativity of diagram \eqref{diagram-selmer}. To check the last assertion, note that $H^1_f(E,V_\p)$ is finite-dimensional over $F_\p$ because $V_\p$ is unramifed outside the finite set $\Xi$ introduced in \eqref{sigma-eq}. \end{proof}

For any number field $E$ define 
\begin{equation} \label{lambda-dfn-eq}
\BK_\p(E):={\rm im}({\rm AJ}_{E})\subset H^1_f(E,A_\p) 
\end{equation}
and
\[ \BBK_\p(E):=\varphi\bigl(\BK_\p(E)\bigr)\otimes_{\mathcal O_\p} F_\p\subset H^1_f(E,V_\p). \] 
If $E$ is Galois over $\Q$ then $\BK_\p(E)$ and $\BBK_\p(E)$ are equipped with $\Gal(E/\Q)$-actions.

\begin{proposition} \label{isom-finite-prop}
There is an isomorphism
\[ \BK_\p(E)/p^m\BK_\p(E)\simeq {\rm im}({\rm AJ}_{E,m}) \]
of finite $\mathcal O_\p/p^m\mathcal O_\p$-modules.
\end{proposition}

\begin{proof} Taking continuous cohomology of the short exact sequence of Galois modules 
\[ 0\longrightarrow A_\p\xrightarrow{p^m}A_\p\longrightarrow A_\p/p^mA_\p\longrightarrow 0, \] 
where the second arrow is the multiplication-by-$p^m$ map and the third arrow is the canonical projection, and using the identification $W_\p[p^m]=A_\p/p^mA_\p$, yields an injection 
\[ i:H^1_\cont(E,A_\p)\otimes_{\cO_\p}\!(\cO_\p/p^m\cO_\p)\;\longmono\;H^1(E,W_\p[p^m]) \] 
of $\mathcal O_\p/p^m\mathcal O_\p$-modules. On the other hand, if $j:H^1_f(E,W_\p[p^m])\hookrightarrow H^1(E,W_\p[p^m])$ denotes the natural inclusion then part (3) of Corollary 
\ref{niziol-coro-2} implies that ${\rm AJ}_{E,m}$ factors through $j$, and therefore the diagram  
\[ \xymatrix@C=35pt{{\CH}^{k/2}\bigl(\tilde\E_N^{k-2}/E\bigr)_0
\otimes\!(\mathcal O_\p/p^m\mathcal O_\p)\ar[r]^-{{\Psi}}\ar[d]^-{{\rm AJ}_{E,m}}&
H^1_\cont(E,A_\p)\otimes_{\mathcal O_\p}\!(\mathcal O_\p/p^m\mathcal O_\p)\ar@{^{(}->}[d]^-i\\
H^1_f(E,W_\p[p^m])\ar@{^{(}->}[r]^-j&H^1(E,W_\p[p^m]),} \] 
where $\Psi$ is the $\mathcal O_\p/p^m\mathcal O_\p$-linear extension of ${\rm AJ}_E$, commutes. Thus ${\rm im}(i\circ{\Psi})$ is equal to ${\rm im}(j\circ{\rm AJ}_{E,m})$, and the injectivity of $i$ and $j$ shows that ${\rm im}({\Psi})\simeq{\rm im}({\rm AJ}_{E,m})$. On the other hand, ${\rm im}(\Psi)=\BK_\p(E)/p^m\BK_\p(E)$, and we are done. \end{proof}

In particular, Proposition \ref{isom-finite-prop}  implies that there is an injection 
\begin{equation} \label{eq13}
\BK_\p(E)/p^m\BK_\p(E)\;\longmono\;H^1_f(E,W_\p[p^m])
\end{equation}
of finite $\mathcal O_\p/p^m\mathcal O_\p$-modules; this map is Galois-equivariant if $E$ is Galois over $\Q$. 

\begin{remark} \label{fin-notation-rem}
By an abuse of notation, we will often adopt the same symbol to denote an element of $\BK_\p(E)/p^m\BK_\p(E)$ and its image in $H^1_f(E,W_\p[p^m])$ via \eqref{eq13}.
\end{remark}

\subsection{Beilinson--Bloch conjecture} 

Now we recall the Beilinson--Bloch conjecture in this setting. Let $E$ be a number field and let $L(f\otimes E,s)$ be the complex $L$-function of $f$ over $E$.
  
\begin{conjecture}[Beilinson--Bloch, \cite{Beil}, \cite{Bloch}] \label{BK} 
$\dim_{F_\p}\bigl(X_\p(E)\bigr)=\ord_{s=\frac{k}{2}}L(f\otimes E,s)$.
\end{conjecture}  

For details, see \cite[pp. 158--168]{Jan}. For generalizations to $L$-functions of motives, see \cite{BK}. The main result of \cite{Nek}, combined with the Gross--Zagier type formula for higher weight modular forms due to Zhang (\cite{Zh}), gives the following result in the direction of the Beilinson--Bloch conjecture.

\begin{theorem}[Nekov\'a\v{r}--Zhang] 
Let $K$ be an imaginary quadratic field in which all the prime numbers dividing $N$ split. If $\ord_{s=\frac{k}{2}}L(f\otimes K,s)=1$ then $\dim_{F_\p}\bigl(X_\p(K)\bigr)=1$. 
\end{theorem}

See \cite[\S 5.3]{Zh} for other results on the Beilinson--Bloch conjecture, especially when the base field is $\Q$. 

\section{Divisibility properties of Heegner cycles}

After reviewing the basic properties of Heegner cycles and the formalism of Darmon--Kolyvagin derivatives, we construct Kolyvagin classes attached to Heegner cycles and study their properties. The main result of this section (Theorem \ref{div-thm}) is a congruence relation satisfied by these cohomology classes.

Fix throughout this paper an imaginary quadratic field $K$ of discriminant $D$ in which all the primes dividing $N$ split (in other words, $K$ satisfies the so-called ``Heegner hypothesis'' relative to $N$). Denote by $\mathcal O_K$ the ring of integers of $K$ and by $h_K$ its class number. For the sake of simplicity, assume also that $\mathcal O_K^\times=\{\pm1\}$, i.e., that $K\not=\Q(\sqrt{-1})$ and $K\not=\Q(\sqrt{-3})$. Finally, fix an embedding $K\hookrightarrow\C$. 

\subsection{Heegner cycles} \label{secHC}

We review construction and basic properties of Heegner cycles on Kuga--Sato varieties. In doing this, we follow \cite{Nek} and \cite{Nek2} closely (for Heegner-type cycles on more general varieties that are fibered over modular curves, see \cite[Section 2]{BDP}).

Fix an ideal $\mathcal N\subset\mathcal O_K$ such that $\mathcal O_K/\mathcal N\simeq \Z/N\Z$, which exists thanks to the Heegner hypothesis satisfied by $K$. For any integer $T\geq1$ prime to $NDp$ let $\mathcal O_T:=\Z+T\mathcal O_K$ be the order of $K$ of conductor $T$. Let $X_0(N)$ be the compact modular curve of level $\Gamma_0(N)$; the isogeny $\C/\cO_T\rightarrow \C/(\cO_T\cap\mathcal N)^{-1}$ defines a Heegner point $x_T\in X_0(N)$ that, by the theory of complex multiplication, is rational over the ring class field $K_T$ of $K$ of conductor $T$ (in particular, $K_1$ is the Hilbert class field of $K$). 

Write $\kappa:X_N\rightarrow X_0(N)$ for the map induced by the inclusion $\Gamma(N)\subset\Gamma_0(N)$ and choose $\tilde x_T\in \kappa^{-1}(x_T)$. The elliptic curve $E_T$ corresponding to $\tilde x_T$ has complex multiplication by $\mathcal O_T$. Fix the unique square root $\xi_T=\sqrt{-DT^2}$ of the discriminant of $\mathcal O_T$ with positive imaginary part under the chosen embedding of $K$ into $\C$. For any $a\in \mathcal O_T$ let $\Gamma_{T,a}\subset E_T\times E_T$ denote the graph of $a$ and let $i_{\tilde x_T}:\pi_{k-2}^{-1}(\tilde x_T)\hookrightarrow \tilde{\mathcal E}_N^{k-2}$ be the canonical inclusion. Then 
\begin{equation} \label{cycle-eq1}
\Pi_B\Pi_\epsilon (i_{\tilde x_T})_*\Big(\Gamma_{T,\xi_T}^{(k-2)/2}\Big)\in\Pi_B\Pi_\epsilon\Big({\CH}^{k/2}(\tilde\E_N^r/K_T)\otimes\Z_p\Big) 
\end{equation} 
and we define the \emph{Heegner cycle}
\[ y_{T,\p}\in H^1_\cont(K_T,A_\p) \] 
to be the image of the cycle in \eqref{cycle-eq1} via the map ${\Psi}_{f,\p,K_T}$ introduced in \eqref{AJ-map}. This class is independent of the choice of $\tilde x_T$ (\cite[p. 107]{Nek}) and, by \cite[Ch. II, \S 3.6]{Nek2}, does not change if $\Gamma_{T,\xi_T}$ is replaced by $\Gamma_{T,\xi_T}\smallsetminus[(E_T\times\{0\})\cup(\{0\}\times E_T)]$ in \eqref{cycle-eq1}, which is the choice made in \cite[\S 5]{Nek}. Finally, note that 
\[ y_{T,\p}\in\BK_\p(K_T) \] 
because the Abel--Jacobi map ${\rm AJ}_{K_T}$ factors through $\Psi_{f,\p,K_T}$. 

Define 
\begin{equation} \label{S-set-eq}
\mathcal S:=\bigl\{\text{$\ell$ prime number}\;\mid\;\text{$\ell$ is inert in $K$ and $\ell\nmid Np$}\bigr\}. 
\end{equation} 
For each $\ell\in\mathcal S$ the extension $K_\ell/K_1$ is cyclic of order $\ell+1$ and unramified at primes different from $\ell$. Also, if $\ell\neq\ell'$ are in $\mathcal S$ then $K_\ell$ and $K_{\ell'}$ are linearly disjoint over $K_1$. Fix a product $T=\prod_{i=1}^s\ell_i$ of distinct primes $\ell_i\in\mathcal S$, then put $G_T:=\Gal(K_T/K_1)$ and $\Gamma_T:=\Gal(K_T/K)$. The field $K_T$ is the composite of the fields $K_{\ell_i}$, which are linearly disjoint over $K_1$, and so there is a decomposition $G_T=\prod_{i=1}^sG_{\ell_i}$. In particular, if $T'\,|\,T$ then there is a canonical inclusion $G_{T'}\subset G_T$, using which we identify the elements of $G_{T'}$ with their images in $G_T$. Finally, set $\Gamma_1:=\Gal(K_1/K)$, so that $\Gamma_1\simeq\Pic(\cO_K)$ and $|\Gamma_1|=h_K$. 

Let us recall two basic properties of Heegner cycles, which extend those of Heegner points and are due to Nekov\'a\v{r} (\cite{Nek}). Before stating them, we fix some notations that will be used in the rest of the paper.
 
Choose a complex conjugation $c\in G_\Q$ and use the same symbol to denote the images of $c$ in quotients of $G_\Q$; in other words, $c$ is a lift to $G_\Q$ of the generator of $\Gal(K/\Q)$. We shall also write $\Frob_\infty$ for the conjugacy class of $c$ in $\Gal(E/\Q)$, relying on the context to make clear which number field $E$ we are considering. Finally, recall that $\cores_{K_{T\ell}/K_T}$ denotes the corestriction map from $H^1(K_{T\ell},A_\p)$ to $H^1(K_T,A_\p)$ and let $\epsilon$ be the sign of the functional equation of $L(f,s)$.

\begin{proposition} \label{complex-conj}
Let $T$ be a square-free product of primes in $\mathcal S$.
\begin{enumerate}
\item If $\ell\in\mathcal S$, $\ell\nmid T$ then $\cores_{K_{T\ell}/K_T} (y_{T\ell,\p})=(a_\ell/\ell^{k/2-1})\cdot y_{T,\p}$.
\item There exists $\sigma\in\Gamma_T$ such that $c(y_{T,\p})=-\epsilon\cdot\sigma(y_{T,\p})$.
\end{enumerate}
\end{proposition}

\begin{proof} Upon applying the projection $A_p\twoheadrightarrow A_\p$, part (1) is \cite[Proposition 6.1, (1)]{Nek}, while part (2) is \cite[Proposition 6.2]{Nek}. (Note the misprint in \emph{loc. cit.}, since the Hecke action is twisted by $k/2-1$.) \end{proof}

\begin{remark}
The relations stated in Proposition \ref{complex-conj}, together with the Key Formula appearing in \cite[\S 9]{Nek} (that will be used in the proof of Proposition \ref{finite-2-prop} below), describe an \emph{Euler system} for modular forms of weight $k>2$. Euler systems for higher weight modular forms can also be constructed by using Howard's work \cite{Ho} on the variation of Heegner points in Hida families, later extended to the case of indefinite Shimura curves in \cite{Fou} and \cite{LV-Man}, by specialization to weight $k$. The relation between the two systems has been investigated by Castella in \cite{Cas}, and we expect that a similar approach could be adopted in the case of indefinite Shimura curves as well. We finally remark that, in yet another direction, 
it would be interesting to generalize to higher weight the Euler systems of Heegner points introduced by means of congruences between modular forms in \cite{BD-Iwasawa} and developed in \cite{Lo1}, \cite{Lo2}, \cite{Lo3}, \cite{LRV}, \cite{LV-JNT}, \cite{Nek-Level}. In connection with this, see work in progress by Chida and Hsieh (\cite{CH}). 
\end{remark}

\subsection{$\pm$-eigenspaces} \label{eigenspaces-subsec}

Recall that if $M$ is an abelian group endowed with an action of an involution $\tau$ and $2$ is invertible in $\End(M)$ then there is a decomposition $M=M^+\oplus M^-$ where $M^\pm$ is the subgroup of $M$ on which $\tau$ acts as $\pm1$. 
 
Let  $p$ be a prime number as in the introduction and let $\p$ a prime ideal of $\mathcal O_F$ above $p$. Since $\Gal(K/\Q)$ acts on $\BBK_\p(K)$, the above formalism applies and there is a decomposition 
\[ \BBK_\p(K)=\BBK_\p(K)^+\oplus\BBK_\p(K)^-. \] 
Define $\rho^\pm_\p:=\dim_{F_\p}\bigl(\BBK_\p(K)^\pm\bigr)$ and
\begin{equation} \label{rho-eq}
\rho_\p:=\begin{cases}\max\bigl\{\rho^+_\p,\rho^-_\p\bigr\}-1&\text{if $\rho^+_\p\not=\rho^-_\p$},\\[3mm]\rho^+_\p&\text{otherwise}.\end{cases} 
\end{equation} 
Two remarks on these definitions, both related with the Beilinson--Bloch conjecture, are now in order.  

\begin{remark}  
1) Conjecture \ref{BK} predicts, among other things, that the $F_\p$-dimension of $\BBK_\p(E)$ does not depend on $\p$, and therefore $\rho_\p^++\rho_\p^-$ is conjecturally independent of $\p$. Moreover, let $f\otimes\epsilon_K$ be the twist of $f$ by the quadratic Dirichlet character $\epsilon_K$ attached to the extension $K/\Q$. It can be shown (see \cite[\S 6.1]{LV} for details; in \cite{LV} a $p$-ordinarity assumption is made, but this condition plays no role in the results about Selmer groups that we are interested in) that 
\[ \BBK_\p(K)^+\simeq\BBK_\p(\Q)={\rm im}({\Psi}_{f,\p,\Q})\otimes_{\cO_\p}\!F_\p,\qquad\BBK_\p(K)^-\simeq{\rm im}({\Psi}_{f\otimes\epsilon_K,\p,\Q})\otimes_{\cO_\p}\!F_\p. \]
Therefore Conjecture \ref{BK} (for $f$ and $E=\Q$ or $f\otimes\epsilon_K$ and $E=\Q$) implies that $\rho_\p^+$ and $\rho_\p^-$ do not depend on $p$.

2) As before, let $L(f\otimes K,s)$ denote the $L$-function of $f$ over $K$, so that
\begin{equation} \label{L-eq} 
L(f\otimes K,s)=L(f,s)\cdot L(f\otimes\epsilon_K,s). 
\end{equation}
Since the orders of vanishing of $L(f,s)$ and $L(f\otimes\epsilon_K,s)$ at $s=k/2$ have opposite parities (cf., e.g., \cite[p. 543]{BFH}), it follows from \eqref{L-eq} that $L(f\otimes K,s)$ vanishes to odd order at $s=k/2$. Therefore Conjecture \ref{BK} predicts that the $F_\p$-dimension $\rho_\p^++\rho_\p^-$ of $\BBK_\p(K)$ should be odd, hence we expect the second possibility in \eqref{rho-eq} not to occur.
\end{remark}

\subsection{Rank inequalities}

As a consequence of the structure theorem for finitely generated modules over principal ideal domains, a finite $\mathcal O_\p/p^m\mathcal O_\p$-module $M$ 
can be decomposed as
\begin{equation} \label{r-selmer-eq}
M\simeq(\mathcal O_\p/p^m\mathcal O_\p)^{r_{\p,m}(M)}\oplus\tilde M 
\end{equation} 
where the exponent of $\tilde M$ divides $p^m$ strictly and the integer $r_{\p,m}(M)$ does not depend on such a decomposition (see Lemma \ref{basic-lemma} below).

Let $\mathbb F_\p:=\cO_\p/p\cO_\p$ be the residue field of $\cO_\p$. In the sequel we will make use of the following auxiliary result.

\begin{lemma} \label{basic-lemma}
Let $M, M', M''$ be finite $\mathcal O_\p/p^m\mathcal O_\p$-modules. 
\begin{enumerate}
\item If there is an injective homomorphism $M\hookrightarrow M'$ then $r_{\p,m}(M)\leq r_{\p,m}(M')$.
\item If there is a surjective homomorphism $M\twoheadrightarrow M'$ then $r_{\p,m}(M)\geq r_{\p,m}(M')$.
\item If there is an exact sequence of $\mathcal O_\p/p^m\mathcal O_\p$-modules
\[ 0\longrightarrow M'\longrightarrow M\longrightarrow M'' \] 
then 
\[ r_{\p,m}(M)\leq r_{\p,m}(M')+\dim_{\mathbb F_\p}(M''\otimes_{\mathcal O_\p/p^m\mathcal O_\p}\!\mathbb F_\p). \]
\end{enumerate}
\end{lemma}

\begin{proof}
An injection $M\hookrightarrow M'$ of $\mathcal O_\p/p^m\mathcal O_\p$-modules induces an injection
$p^{m-1}M\hookrightarrow p^{m-1}M'$ of $\F_\p$-vector spaces, hence
\[ r_{\p,m}(M)=\dim_{\F_\p}(p^{m-1}M)\leq \dim_{\F_\p}(p^{m-1}M')=r_{\p,m}(M'), \]
which shows part (1). On the other hand, a surjection $M\twoheadrightarrow M'$ of $\mathcal O_\p/p^m\mathcal O_\p$-modules induces a surjection $p^{m-1}M\twoheadrightarrow p^{m-1}M'$ of $\F_\p$-vector spaces, and part (2) follows similarly. Finally, part (3) can be proved as \cite[Lemma 5.1]{Dar}. \end{proof}

As before, let $K$ be our imaginary quadratic field where all the prime factors of $N$ split. With notation as in \eqref{r-selmer-eq}, set
\[ \tilde r_{\p,m}:=r_{\p,m}\bigl(H^1_f(K,W_\p[p^m])\bigr).\] 
Moreover, recall the integers $\rho_\p^\pm$ introduced in \S \ref{eigenspaces-subsec} and define 
\begin{equation} \label{tilde-rho-eq}
\tilde\rho_\p:=\rho_\p^++\rho_\p^-=\dim_{F_\p}\bigl(\BBK_\p(K)\bigr). 
\end{equation}
Observe that there is an obvious inequality
\begin{equation} \label{ineq-1-eq}
\tilde\rho_\p\leq r_{\p,m}\bigl({\BK}_\p(K)/p^m{\BK}_\p (K)\bigr).
\end{equation}

\begin{proposition} \label{lemSelmer} 
$\tilde\rho_\p\leq \tilde r_{\p,m}$. 
\end{proposition}

\begin{proof} It follows from Proposition \ref{isom-finite-prop} and Lemma \ref{basic-lemma} that  
\begin{equation} \label{eq-2}
r_{\p,m}\bigl({\BK}_\p(K)/p^m{\BK}_\p(K)\bigr)=r_{\p,m}\bigl({\rm im}({\rm AJ}_{K,m})\bigr)\leq r_{\p,m}\bigl(H^1_f(K,W_\p[p^m])\bigr)=\tilde r_{\p,m}.
\end{equation}
Combining \eqref{ineq-1-eq} and \eqref{eq-2} gives the desired inequality. \end{proof}

Since $p$ is odd, there is a splitting 
\[ H^1_f(K,W_\p[p^m])=H^1_f(K,W_\p[p^m])^+\oplus H^1_f(K,W_\p[p^m])^- \]
under the action of complex conjugation $c\in\Gal(K/\Q)$. Set 
\[ \tilde r^\pm_{\p,m} :=r_{\p,m}\bigl(H^1_f(K,W_\p[p^m])^\pm\bigr) \]
and define 
\[ r_{\p,m} :=\begin{cases}\max\bigl\{\tilde r^+_{\p,m},\tilde r^-_{\p,m}\bigr\}-1&\text{if $\tilde r^+_{\p,m}\not=\tilde r^-_{\p,m}$},\\[3mm]\tilde r^+_{\p,m} &\text{otherwise}.\end{cases} \]
Recall the integer $\rho_\p$ defined in \eqref{rho-eq}.

\begin{proposition} \label{lem4.2}
$\rho_\p\leq r_{\p,m}$. 
\end{proposition}

\begin{proof} Combine the $\Gal(K/\Q)$-equivariance of the Abel--Jacobi map with Proposition \ref{lemSelmer}. \end{proof}

\subsection{Darmon--Kolyvagin derivatives} \label{kolyvagin-subsec}

In this subsection we consider the general formalism of Darmon--Kolyvagin derivatives in the case of ring class fields of square-free conductor. 

Fix a square-free product $S=\prod_{i=1}^t\ell_i$ of primes $\ell_i$ in $\mathcal S_{p^m}$. For a prime $\ell\,|\,S$ let $\sigma_{\ell}$ be a generator of $G_{\ell}$. For any integer $k$ such that $0\leq k\leq \ell=\#G_\ell-1$ define the derivative operator 
\[ {\bf D}_\ell^k:=\sum_{i=k}^{\ell}\binom ik\sigma_\ell^i\in\Z[G_\ell]\subset\cO_\p[G_\ell]. \]
If $\kappa=(k_1,\dots,k_t)\in\Z^t$ with $0\leq k_i\leq \ell_i$ then the \emph{Darmon--Kolyvagin $\kappa$-derivative} is 
\[ {\bf D}_{{\kappa}}:={\bf D}_{\ell_1}^{k_1}\cdots{\bf D}_{\ell_t}^{k_t}\in\Z[G_S]\subset\cO_\p[G_S]. \]  
The \emph{order}, the \emph{support} and the \emph{conductor} of ${\bf D}_{{\kappa}}$ are defined as
\[ \ord({\bf D}_{{\kappa}}):=\sum_{i=1}^tk_i,\quad{\rm supp}({\bf D}_{{\kappa}}):=S,\quad{\rm cond}({\bf D}_{{\kappa}}):=\prod_{k_i>0}\ell_i, \]
respectively, and we set
\[ \eta(\kappa):=\min\bigl\{\ord_p(n_i)\;\mid\;k_i>0\bigr\}. \]
Finally, given ${\kappa} =(k_1,\dots,k_s)$ and ${\kappa}'=(k_1',\dots,k_s')$ we say that ${\bf D}_{{\kappa}'}$ is \emph{less than} ${\bf D}_{\kappa}$ if $k_i'\leq k_i$ for all $i$, and we write $\kappa'\leq\kappa$ in this case. Moreover, we say that ${\bf D}_{\kappa'}$ is \emph{strictly less than} ${\bf D}_\kappa$, written $\kappa'<\kappa$, if $\kappa'\leq\kappa$ and $\kappa'\neq\kappa$. 

Now we collect some basic facts about these derivatives. Most of them will not be used until Section \ref{secMR}, but we prefer to gather them here for the sake of clarity. The proofs are straightforward computations and will be omitted: see \cite[\S 3.1 and \S 4.1]{Dar} for details.  

\subsubsection{Taylor's formula} \label{3.3.1} 

The \emph{resolvent element} associated with an element $m$ of an $\cO_\p[G_S]$-module $M$ is defined as 
\[ \theta_m:=\sum_{\sigma\in G_S}\sigma(m)\otimes \sigma\in M\otimes_{\cO_\p}\!\cO_\p[G_S]. \] 
Then 
\[ \theta_m=\sum_\kappa {\bf D}_\kappa(m)\otimes(\sigma_1-1)^{k_1}\dots(\sigma_t-1)^{k_t}, \] where the sum is taken over all $t$-tuples of integers $\kappa=(k_1,\dots,k_t)$, with the convention that only those $\kappa$ with $0\leq k_i\leq\ell_i$ for all $i$ appear in the above sum.  

\subsubsection{Divisibility criterion} \label{3.3.2} 

Let $I_{G_S}$ be the augmentation ideal of $\mathcal O_\p[G_S]$ and let $r\leq p$ be an integer. If ${\bf D}_\kappa(m)\equiv0\pmod{p^{\eta(\kappa)}}$ for all $\kappa$ with $\ord(\kappa)<r$ then $\theta_m$ belongs to the natural image of $M\otimes_{\cO_\p}\!I_{G_S}^r$.  

\subsubsection{Action of complex conjugation} \label{3.3.3} 

The action of $c\in\Gal(K/\Q)$ on $\Gamma_S=\Gal(K_S/K)$ by conjugation sends $\sigma$ to $\sigma^{-1}$. This induces an action of $c$ on $\cO_\p[G_S]$ by linearity, and the formula
\[ c\,{\bf D}_\kappa\,c^{-1}=(-1)^{\ord({\bf D}_\kappa)}{\bf D}_\kappa+\sum_{ \kappa'<\kappa}\alpha_{\kappa'}{\bf D}_{\kappa'} \] 
holds for suitable integers $\alpha_{\kappa'}$. 

\subsubsection{Some formulas} \label{3.3.4}

For any prime $\ell\,|\,S$ and any integer $k$ with $0\leq k\leq\ell$ we have 
\[ (\sigma_\ell-1){\bf D}_\ell^k=\binom {\ell+1}k-\sigma_\ell{\bf D}_\ell^{k-1}. \] 
In particular, since $p^m\,|\,\ell+1$, for all $0<k<p$ we have 
\begin{equation} \label{sigma-ell-D-eq}
(\sigma_\ell-1){\bf D}_\ell^k\equiv -\sigma_\ell{\bf D}^{k-1}_\ell\pmod{p^m}. 
\end{equation}

\subsubsection{Special bases} \label{3.2.5} 

An element $\xi\in\Z[G_\ell]$, for a prime $\ell\,|\,S$, can be written as a $\Z$-linear combination of the derivatives ${\bf D}_\ell^k$ for $k=0,\dots,\ell$. Since this is not justified in \cite{Dar}, we give a short proof. Write $\xi=\sum_{i=0}^\ell a_i\sigma_\ell^i$. By rearranging the sums, one can check that a linear combination $\sum_{k=0}^\ell \alpha_k{\bf D}_\ell^k$ of derivatives can be written as $\sum_{i=0}^\ell\!\Big(\!\sum_{k=0}^i\alpha_k\binom ik\!\Big)\sigma^i_\ell$. Therefore we have to prove that we can find coefficients $\alpha_k\in\Z$ such that $\sum_{k=0}^i\alpha_k\binom ik=a_i$ for all $i=0,\dots,\ell$. The generic equation in this system is 
\[ \alpha_0+i\alpha_1+\binom i2\alpha_2+\dots+\binom i{i-1}\alpha_{i-1}+\alpha_i=
a_i, \] 
and the desired solution can be found recursively.

\subsection{The set of exceptional primes} \label{exc-set}

The main result of this section, Theorem \ref{div-thm}, applies to all primes $p$ outside a finite set $\Sigma$ that we describe below. 

Let $\Sigma$ be the set of prime numbers $p$ satisfying at least one of the following conditions:   
\begin{itemize} 
\item $p\,|\,6ND(k-2)!\phi(N)$ and $p$ ramifies in $F$; 
\item the image of the $p$-adic representation 
\[ \rho_{f,p}:G_\Q\longrightarrow\GL_2(\cO_F\otimes\Z_p) \] 
attached to $f$ by Deligne (\cite{Del})  does not contain the set 
\[ \bigl\{g\in\GL_2(\cO_F\otimes\Z_p)\;\mid\;\det(g)\in(\Z_p^\times)^{k-1}\bigr\}. \]  
\end{itemize}
To begin with, we need the following

\begin{lemma} \label{lemma3.3}
The set $\Sigma$ is finite. 
\end{lemma}

\begin{proof} The only non-trivial fact to check is that there are only finitely many prime numbers satisfying the last condition, and this follows from \cite[Theorem 3.1]{Rib}. \end{proof}

For a prime number $p\not\in\Sigma$ and an integer $m\geq1$ define 
\begin{equation} \label{def-S}
\mathcal S_{p^m}:=\bigl\{\text{$\ell$ prime number}\; \mid\; \text{$\ell$ is inert in $K$, $\ell\nmid N$ and $p^m\,|\,\ell+1$}\bigr\}.
\end{equation}
Notice that $\mathcal S_{p^m}\subset\mathcal S$ with $\mathcal S$ defined in \eqref{S-set-eq}. As a piece of notation, when we write that a prime ideal of $\Z$ belongs to a set $\Theta$ of prime numbers we mean that the positive generator of this ideal belongs to $\Theta$. Let $\boldsymbol\mu_{p^m}$ denote the $p^m$-th roots of unity in $\bar\Q$. By \cite[Lemma 3.14]{Dar}, a prime $\ell$ belongs to $\mathcal S_{p^m}$ precisely when $\Frob_\ell=\Frob_\infty$ in $\Gal(K(\boldsymbol\mu_{p^m})/\Q)$, hence $\mathcal S_{p^m}$ is infinite by \v{C}ebotarev's density theorem. Furthermore, there is an inclusion $\boldsymbol\mu_{p^m}\subset K_\lambda$ for every prime $\lambda$ of $K$ such that $\lambda\cap\Z\in\mathcal S_{p^m}$.

With $\Sigma$ as above, fix from now to the end of this section a prime number $p\not\in\Sigma$ and a quadruplet $(\p^m,S,{\bf D}_\kappa,\ell)$ consisting of
\begin{itemize}
\item a prime ideal $\p$ of $\mathcal O_F$ above $p$;  
\item an integer $m\geq1$; 
\item a square-free product $S=\prod_i\ell_i$ of primes $\ell_i$ in the set $\mathcal S_{p^m}$ introduced in \eqref{def-S};
\item a derivative ${\bf D}_\kappa$ with ${\rm supp}({\bf D}_\kappa)=S$;
\item an auxiliary prime $\ell\in\mathcal S_{p^m}$.
\end{itemize} 
\subsection{Kolyvagin classes attached to Heegner cycles} \label{heegner-classes-subsec}

In this subsection we introduce classes $d(\ell)\in H^1(K,W_\p[p^m])$ depending on the data $S$, $p^m$, ${\bf D}_\kappa$ and $\ell$.

Let
\[ \vartheta_\p:G_\Q\longrightarrow\Aut(A_\p) \] 
be the Galois representation attached to $A_\p$. For every integer $m\geq1$ the group $G_\Q$ acts on $W_\p[p^m]$ via its action on $A_\p$, and we obtain a representation
\[ \bar\vartheta_{\p,m}:G_\Q\longrightarrow\Aut(W_\p[p^m]). \]
In particular, $\bar\vartheta_\p:=\bar\vartheta_{\p,1}$ is a residual representation of $G_\Q$ over the finite field $\F_\p$.

For any subfield $L$ of $\bar\Q$ write $A_\p(L)$ and $W_\p[p^m](L)$ for $H^0_{\rm cont}(L,A_\p)$ and $H^0(L,W_\p[p^m])$, respectively; similar conventions apply when $L$ is a completion of a number field. 

\begin{lemma} \label{lem-irr}
If $\p\cap\Z\not\in\Sigma$ then $\vartheta_\p$ and $\bar\vartheta_\p$ are irreducible and have non-solvable images.
\end{lemma}

\begin{proof} By \cite[Lemma 6.2]{Bes}, the image of $\vartheta_\p$ in $\Aut(A_\p)\simeq\GL_2(\mathcal O_\p)$ contains a subgroup that is conjugate to $\GL_2(\Z_p)$. Since $p\neq2$, this implies the irreducibility of $\vartheta_\p$ and $\bar\vartheta_\p$ (see \cite[Proposition 6.3, (1)]{Bes} for details). Finally, the groups $\GL_2(\Z_p)$ and $\GL_2(\F_p)$ are not solvable because $p>3$, hence the images of $\vartheta_\p$ and $\bar\vartheta_\p$ are not solvable. \end{proof} 

\begin{lemma} \label{lem-sol}
If $\p\cap\Z\not\in\Sigma$ and the extension $E/\Q$ is solvable then 
\begin{enumerate}
\item $A_\p(E)=0$;
\item $W_\p[p^n](E)=0$ for all $n\geq1$.
\end{enumerate} 
\end{lemma}

\begin{proof} Let us prove part (1). Since $\p\cap\Z\not\in\Sigma$, Lemma \ref{lem-irr} ensures that $\vartheta_\p$ is irreducible with non-solvable image. The submodule $A_\p(E)$ of $A_\p$ is $G_\Q$-stable, hence if $A_\p(E)\neq0$ then $A_\p(E)=A_\p$ by the irreducibility of $\vartheta_\p$. Thus $\vartheta_\p$ factors through $\Gal(E/\Q)$, which is solvable by assumption. It follows that $\text{im}(\vartheta_\p)$ is solvable, which is a contradiction. Finally, in order to prove part (2) it is of course enough to prove the claim for $n=1$, and this can be done \emph{mutatis mutandis} in the same way, using again Lemma \ref{lem-irr}. \end{proof}

Write $L_m:=K(W_\p[p^m])$ for the composite of $K$ and the subfield of $\bar\Q$ fixed by $\ker(\bar\vartheta_{\p,m})$. With notation as in \S \ref{secHC}, define a set $ \tilde{\mathcal S}_{p^m}$ of prime numbers as
\[ \tilde{\mathcal S}_{p^m}:=\bigl\{\text{$\ell$ prime number}\;\;\big|\;\;\text{$\ell\nmid NDp$ and $\Frob_\ell=\Frob_\infty$ in $\Gal(L_m/\Q)$}\bigr\}. \]
Again by \v{C}ebotarev's density theorem, $ \tilde{\mathcal S}_{p^m}$ is infinite.
  
\begin{lemma}\label{lemma5.2}
A prime $\ell$ not dividing $D$ belongs to $ \tilde{\mathcal S}_{p^m}$ if and only if $\ell$ belongs to $\mathcal S_{p^m}$ and $\p^m$ divides $a_\ell$ in $\mathcal O_F$. 
\end{lemma}

\begin{proof} Equating the minimal polynomials of $F_\ell$ (see \eqref{char-pol}) and of $c$ acting on $W_\p[p^m]$, one finds the divisibility relations $\p^m\,|\,a_\ell$ and $\p^m\,|\,\ell+1$ in $\cO_F$. Since $p$ is unramified in $F$, the second relation gives an inclusion $(\ell+1)\subset(p^m)$ of principal ideals of $\cO_F$; this immediately implies that $p^m\,|\,\ell+1$ in $\Z$, which concludes the proof. \end{proof}

With notation as before, let $\ell\in \tilde{\mathcal S}_{p^m}$ and put $T:=S\ell$. Define 
\[ \tilde P(\ell):={\bf D}_{{\kappa}}{\bf D}_\ell^1(y_{T,\p})\in \BK_\p(K_T), \]  
then denote by 
\[ P(\ell) \in\BK_\p(K_T)/p^m\BK_\p(K_T) \]
the image of $\tilde P(\ell)$ under the canonical projection. 

With the exception of \S \ref{tate-subsec}, from here till the end of \S \ref{local-kolyvagin-subsec} we will work under the following technical assumption on $(\p^m,S,{\bf D}_{{\kappa}},\ell)$.

\begin{assumption} \label{ass} 
For all ${\bf D}_{{\kappa}'}$ strictly less than ${\bf D}_{{\kappa}}{\bf D}_\ell^1$ we have ${\bf D}_{{\kappa}'}(y_{T,\p})=0$. 
\end{assumption}

With this condition in force, we can prove

\begin{lemma} \label{lemma5.3}
The class $P(\ell)$ is fixed by the action of $G_T$. 
\end{lemma}

\begin{proof} Let $\sigma=\sigma_{\ell_i}$ or $\sigma=\sigma_\ell$. Congruence \eqref{sigma-ell-D-eq} shows that 
\[ (\sigma-1)\tilde P(\ell)\equiv-\sigma{\bf D}_{\kappa'}(y_{T,\p})\pmod{p^m} \]
for some ${\bf D}_{\kappa'}$ strictly less than ${\bf D}_\kappa{\bf D}^1_\ell$, which concludes the proof. \end{proof}

Recall from \eqref{eq13} that there is an injective, Galois-equivariant map of $\mathcal O_\p/p^m\mathcal O_\p$-modules 
\[ \BK_\p(K_T)/p^m\BK_\p(K_T)\;\longmono\;H^1_f(K_T,W_\p[p^m])\subset H^1(K_T,W_\p[p^m]). \]
By Lemma \ref{lemma5.3}, the image of $P(\ell)$ via this map belongs to $H^1_f(K_T,W_\p[p^m])^{G_T}$, hence to $H^1(K_T,W_\p[p^m])^{G_T}$. Since $K_T/\Q$, being generalized dihedral, is solvable, part (2) of Lemma \ref{lem-sol} and the inflation-restriction exact sequence give an isomorphism 
\[ \res_{K_T/K_1}:H^1(K_1,W_\p[p^m])\overset\simeq\longrightarrow H^1(K_T,W_\p[p^m])^{G_T}. \] 
Let ${\bf N}={\bf N}_{K_1/K}:=\sum_{\sigma\in\Gamma_1}\sigma\in\Z[\Gamma_1]$ denote the norm operator from $K_1$ to $K$. The abelian group $H^1(K_1,W_\p[p^m])$ is naturally a $\Gamma_1$-module, so ${\bf N}$ induces a map
\[ {\bf N}:H^1(K_1,W_\p[p^m])\longrightarrow H^1(K_1,W_\p[p^m])^{\Gamma_1}. \]
Since $p\nmid h_K$, inflation-restriction shows that there is an isomorphism 
\[ \res_{K_1/K}:H^1(K,W_\p[p^m])\overset{\simeq}\longrightarrow H^1(K_1,W_\p[p^m])^{\Gamma_1}. \]
Consider the diagram
\[ \xymatrix@C=45pt{H^1(K_1,W_\p[p^m])\ar[d]^-{\bf N}&H^1(K_T,W_\p[p^m])^{G_T}\ar[l]_-{\res_{K_T/K_1}^{-1}}\ar@{-->}[d]^-\beta\\H^1(K_1,W_\p[p^m])^{\Gamma_1}\ar[r]^-{\res_{K_1/K}^{-1}}&H^1(K,W_\p[p^m])} \]
where the dotted arrow $\beta$ is defined so as to make the resulting square commute. Thus we can attach to $P(\ell)\in H^1(K_T,W_\p[p^m])^{G_T}$ a \emph{Kolyvagin class} 
\[ d(\ell):=\beta\bigl(P(\ell)\bigr)\in H^1(K,W_\p[p^m]) \] 
such that 
\begin{equation} \label{res-d-eq}
\res_{K_T/K}\bigl(d(\ell)\bigr)={\bf N}_T\bigl(P(\ell)\bigr)
\end{equation}
where ${\bf N}_T\in\Z[\Gamma_T]$ is an arbitrary lift of ${\bf N}$ via the canonical projection $\Gamma_T\twoheadrightarrow\Gamma_1$ (if ${\bf N}'_T$ is another such lift then ${\bf N}_T(P(\ell))= {\bf N}'_T(P(\ell))$ by Lemma \ref{lemma5.3}). Furthermore, since $\res_{K_T/K}$ is an isomorphism, $d(\ell)$ is the only class in $H^1(K,W_\p[p^m])$ satisfying \eqref{res-d-eq}.

\subsection{Action of complex conjugation on Kolyvagin classes}

Recall that $\epsilon$ is the sign of the functional equation of $L(f,s)$ and set $\epsilon_\kappa:=(-1)^{\ord({\bf D}_\kappa)}\epsilon$. 

\begin{proposition} \label{eigen-d-prop}
The class $d(\ell)$ belongs to the $\epsilon_{{\kappa}}$-eigenspace of $H^1(K,W_\p[p^m])$ under the action of $c$.
\end{proposition}

\begin{proof} By \S\ref{3.3.3} and Assumption \ref{ass}, there is an equality
\[ c\bigl(\tilde P(\ell)\bigr)=(-1)^{\ord({\bf D}_\kappa{\bf D}_\ell^1)}{\bf D}_\kappa{\bf D}_\ell^1c(y_{T,\p}). \] 
Since the ring $\Z[\Gamma_T]$ is commutative, part (2) of Proposition \ref{complex-conj} then shows that 
\[ c\bigl(P(\ell)\bigr)=-\epsilon(-1)^{\ord({\bf D}_\kappa{\bf D}_\ell^1)}\sigma\bigl(P(\ell)\bigr) \] 
for a suitable $\sigma\in\Gamma_T$. Applying any lift ${\bf N}_T=\sum_{i=1}^{h_K}\sigma_i\in\Z[\Gamma_T]$ of ${\bf N}$ on both sides gives 
\begin{equation} \label{N_T-c-eq}
{\bf N}_T\bigl(c(P(\ell))\bigr)=-\epsilon(-1)^{\ord({\bf D}_\kappa{\bf D}_\ell^1)}{\bf N}_T\bigl(\sigma(P(\ell))\bigr). 
\end{equation}
Now $\sum_{i=1}^{h_K}\sigma_ic=c\sum_{i=1}^{h_K}\sigma_i^{-1}$. Moreover, since ${\bf N}'_T:=\sum_{i=1}^{h_K}\sigma_i^{-1}$ and ${\bf N}''_T:=\sum_{i=1}^{h_K}\sigma_i\sigma$ are two lifts of ${\bf N}$, equality \eqref{res-d-eq} implies that 
\begin{equation} \label{T'-T''-eq}
{\bf N}'_T\bigl(P(\ell)\bigr)=\res_{K_T/K}\bigl(d(\ell)\bigr)={\bf N}''_T\bigl(P(\ell)\bigr). 
\end{equation}
By definition of $\epsilon_\kappa$, combining \eqref{N_T-c-eq} and \eqref{T'-T''-eq} gives 
\[ c\cdot\res_{K_T/K}\bigl(d(\ell)\bigr)=\epsilon_\kappa\res_{K_T/K}\bigl(d(\ell)\bigr), \]
and the conclusion follows from the $\Gal(K/\Q)$-equivariance of the isomorphism $\res_{K_T/K}$. \end{proof}

\subsection{Tate duality} \label{tate-subsec}

In this subsection we do not suppose that Assumption \ref{ass} holds. Let $\lambda$ denote the unique prime of $K$ above $\ell$. By \cite[Proposition 3.1, (2)]{Nek}, there is a $G_\Q$-equivariant skew-symmetric pairing
\[ [\cdot\,,\cdot]:A_\p\times A_\p\longrightarrow\Z_p(1) \]
such that the induced pairing
\[ {[\cdot\,,\cdot]}_m:W_\p[p^m]\times W_\p[p^m]\longrightarrow\boldsymbol\mu_{p^m} \]
is non-degenerate. With notation as before, combining cup product in cohomology with the map $W_\p[p^m]\otimes W_\p[p^m]\rightarrow\boldsymbol\mu_{p^m}$ induced by ${[\cdot\,,\cdot]}_m$ gives rise to a pairing
\[ {\langle\cdot\,,\cdot\rangle}_\lambda:H^1(K_\lambda,W_\p[p^m])\times H^1(K_\lambda,W_\p[p^m])\longrightarrow H^2(K_\lambda,\boldsymbol\mu_{p^m})=\Z/p^m\Z, \]
with the equality on the right coming from the invariant map of local class field theory. By a result of Tate, this pairing is non-degenerate (cf. \cite[Ch. I, Corollary 2.3]{Milne}). 

Since $A_\p$ is unramified at $\lambda$, the group $H^1_f(K_\lambda,W_\p[p^m])=H^1_{\rm ur}(K_\lambda,W_\p[p^m])$ is its own annihilator in $H^1(K_\lambda,W_\p[p^m])$ under Tate's pairing ${\langle\cdot\,,\cdot\rangle}_\lambda$ (\cite[Lemma 4.4]{Bes}). The \emph{singular part} of the cohomology is then defined via the short exact sequence
\[ 0\longrightarrow H^1_f(K_\lambda,W_\p[p^m])\longrightarrow H^1(K_\lambda,W_\p[p^m])\longrightarrow H^1_{sin}(K_\lambda,W_\p[p^m])\longrightarrow0, \]
and ${\langle\cdot\,,\cdot\rangle}_\lambda$ induces a $\Gal(K/\Q)$-equivariant perfect pairing
\begin{equation} \label{perfect-eq}
{\langle\cdot\,,\cdot\rangle}_\lambda:H^1_f(K_\lambda,W_\p[p^m])\times H^1_{sin}(K_\lambda,W_\p[p^m])\longrightarrow\Z/p^m\Z. 
\end{equation} 
It follows that there are natural identifications
\begin{equation} \label{sin-isom-eq}
H^1_{sin}(K_\lambda,W_\p[p^m])=H^1(K_\lambda^{\rm ur},W_\p[p^m])=\Hom_{\,\rm cont}\bigl(\Gal(\bar K_\lambda/K^{\rm ur}_\lambda),W_\p[p^m]\bigr)
\end{equation}
where $K_\lambda^{\rm ur}$ is the maximal unramified extension of $K_\lambda$. Let $K_\lambda^{\rm t}$ denote the maximal tamely ramified extension of $K_\lambda$. The wild inertia group $\Gal(\bar K_\lambda/K^{\rm t}_\lambda)$ is a pro-$\ell$-group and $\ell\not=p$, hence equalities \eqref{sin-isom-eq} yield a further identification
\begin{equation} \label{sin-isom-eq2}
H^1_{sin}(K_\lambda,W_\p[p^m])=\Hom_{\,\rm cont}\bigl(\Gal(K^{\rm t}_\lambda/K^{\rm ur}_\lambda),W_\p[p^m]\bigr). 
\end{equation} 
Fix a (topological) generator $\tau$ of $\Gal(K^{\rm t}_\lambda/K^{\rm ur}_\lambda)$, so that $\tau$ and a lift to $\Gal(K_\lambda^{\rm t}/K_\lambda)$ of the Frobenius $F_\lambda\in\Gal(K_\lambda^{\rm ur}/K_\lambda)$ generate $\Gal(K_\lambda^{\rm t}/K_\lambda)$ topologically. In light of \eqref{sin-isom-eq2}, evaluating homomorphisms at $\tau$ gives an isomorphism
\[ \alpha_\lambda:H^1_{sin}(K_\lambda,W_\p[p^m])\overset\simeq\longrightarrow W_\p[p^m]. \] 
On the other hand, by \cite[Lemma 6.8]{Bes}, if $\ell\in\tilde{\mathcal S}_{p^m}$ then evaluation at $\Frob_\lambda$ gives a $\Gal(K/\Q)$-equivariant isomorphism
\[ \beta_\lambda:H^1_f(K_\lambda,W_\p[p^m])\overset\simeq\longrightarrow W_\p[p^m]. \]
It follows that for every $\ell\in\tilde{\mathcal S}_{p^m}$ there is an isomorphism
\begin{equation} \label{fin-sin-isom-eq}
\nu_\lambda:=\alpha_\lambda^{-1}\circ\beta_\lambda: H^1_f(K_\lambda,W_\p[p^m])\overset\simeq\longrightarrow H^1_{sin}(K_\lambda,W_\p[p^m])
\end{equation}
of $\mathcal O_\p/p^m\mathcal O_\p$-modules.

As a piece of notation, for a $\Z/p^m\Z$-module $M$ write
\[ M^*:=\Hom(M,\Z/p^m\Z) \]
for the Pontryagin dual of $M$. Note that if $M$ is endowed with a $\Z/p^m\Z$-linear action of an involution $\tau$ then $M^*$ inherits a $\Z/p^m\Z$-linear action of $\tau$ by setting
\[ (\tau\cdot f)(m):=f(\tau\cdot m) \]
for all $f\in M^*$ and all $m\in M$. Letting the superscripts $\pm$ denote the $\pm$-eigenspaces under the actions of $\tau$, there are canonical isomorphisms
\begin{equation} \label{eigenspaces-pm-isom-eq}
(M^*)^\pm\overset\simeq\longrightarrow{(M^\pm)}^{\!*}
\end{equation} 
of $\Z/p^m\Z$-modules.

With this notation in force, the pairing in \eqref{perfect-eq} is equivalent to an isomorphism
\begin{equation} \label{isom-pontryagin-eq}
H^1_{sin}(K_\lambda,W_\p[p^m])\overset\simeq\longrightarrow H^1_f(K_\lambda,W_\p[p^m]{)}^*.
\end{equation}
By composing isomorphism \eqref{isom-pontryagin-eq} with the dual of the natural (localization) map 
\[ H^1_f(K,W_\p[p^m])\longrightarrow H^1_f(K_\lambda,W_\p[p^m]), \] 
we obtain a map
\[ \phi_\lambda:H^1_{sin}(K_\lambda,W_\p[p^m])\longrightarrow H^1_f(K,W_\p[p^m]{)}^*. \]
Analogously, for every $\Z/p^m\Z$-submodule $\mathscr S\subset H^1_f(K,W_\p[p^m])$ we obtain by restriction a map $H^1_{sin}(K_\lambda,W_\p[p^m])\rightarrow\mathscr S^*$, which will be denoted by the same symbol. Observe that $\phi_\lambda$ is $\Gal(K/\Q)$-equivariant.

\begin{remark} \label{singular-rem}
In light of the perfect pairing \eqref{perfect-eq}, when dealing with Tate's duality we shall often use the same symbol to denote $d(\ell)_\lambda$ and its image in $H^1_{sin}(K_\lambda,W_\p[p^m])$.
\end{remark}

\subsection{Local behaviour of Kolyvagin classes} \label{local-kolyvagin-subsec}

By class field theory, $\lambda$ splits completely in $K_S/K$; choose a prime $\lambda_S$ of $K_S$ above $\lambda$. Furthermore, $\lambda_S$ is totally ramified in $K_T/K_S$; write $\lambda_T$ for the unique prime of $K_T$ above it. 

As before, if $v$ is a place of $K$ then write $K_v$ for the completion of $K$ at $v$. There is a localization (restriction) map 
\[ \res_v:H^1(K,W_\p[p^m])\longrightarrow H^1(K_v,W_\p[p^m]) \] 
and if $s\in H^1(K,W_\p[p^m])$ then we write $s_v$ for  $\res_v(s)$. 

\begin{proposition}\label{prop5.9}
If $v$ is an archimedean place of $K$ then $d(\ell)_v=0$. 
\end{proposition}

\begin{proof} The quadratic field $K$ is imaginary, hence $K_v=\C$. The proposition follows because $\C$ is algebraically closed and so $H^1(\C,W_\p[p^m])=0$. \end{proof}

Now set $S':={\rm cond}({\bf D}_\kappa)$.

\begin{proposition} \label{finite-1-prop}
If $v$ is a finite place of $K$ not dividing $S'\ell$ then \[d(\ell)_v\in H^1_f(K_v,W_\p[p^m]).\] 
\end{proposition}

\begin{proof} By construction, $P(\ell)$ belongs to $H^1_f(K_{S'\ell},W_\p[p^m])$. If $v\nmid p$ is a prime of $K$ and $v'$ is a prime of $K_{S'\ell}$ above it then $P(\ell)$ belongs to $H^1_{\rm un}(K_{S'\ell,v'},W_\p[p^m])$. By definition, the restriction of $d(\ell)$ is $P(\ell)$. In particular, the restriction of $d(\ell)_v$ under the map 
\[ H^1(K_v,W_\p[p^m])\longrightarrow H^1(K_{S'\ell,v'},W_\p[p^m]) \]
lies in $H^1_{\rm un}(K_{S'\ell,v'},W_\p[p^m])$. By inflation-restriction, the kernel of the above map is 
\[ H^1\bigl(K_{S'\ell,v'}/K_v, W_\p[p^m](K_{S'\ell,v'})\bigr), \] 
and the extension $K_{S'\ell,v'}/K_v$ is unramified, therefore $d(\ell)_v$ is unramified too. On the other hand, if $v\,|\,p$ then the claim follows from the de Rham conjecture for open varieties (now a theorem of Faltings), as explained in \cite[Lemma 11.1, (2)]{Nek}. \end{proof}

Now we begin the study of $d(\ell)_\lambda$ (recall that $\lambda$ is the unique prime of $K$ above the prime number $\ell\in \tilde{\mathcal S}_{p^m}$). For this, we need some preliminaries. 

\begin{lemma} \label{dimension} 
If $\ell\in\tilde{\mathcal S}_{p^m}$ then there are isomorphisms of $\cO_\p$-modules
\[ H^1_{sin}(K_\lambda,W_\p[p^m])^\pm\simeq\cO_\p/p^m\cO_\p,\qquad H^1_f(K_\lambda,W_\p[p^m])^\pm\simeq\cO_\p/p^m\cO_\p. \]
\end{lemma}

\begin{proof} This is \cite[Lemma 6.9, (2)]{Bes}. (Notice that the first three conditions listed on \cite[p. 36]{Bes} are equivalent to the condition $\ell\in\tilde{\mathcal S}_{p^m}$ by \cite[Remark 3.1]{Bes} and that the fourth condition on \cite[p. 36]{Bes} plays no role in the proof of \cite[Lemma 6.9]{Bes}). \end{proof} 

\begin{lemma} \label{frob-invertible-lemma}
If $\ell\in\tilde{\mathcal S}_{p^m}$ then the expressions $\bigl(a_\ell\pm(\ell+1)\Frob_\lambda\bigr)\big/p^m$ define endomorphisms of $H^1_{sin}(K_\lambda, W_\p[p^m])$. Furthermore, if $\bigl(a_\ell\pm(\ell+1)\bigr)\big/p^m$ are both $p$-adic units then the above maps are invertible.
\end{lemma}

\begin{proof} Since $\ell\in\tilde{\mathcal S}_{p^m}$, there is an equality $\Frob_\ell=\Frob_\infty$ of conjugacy classes in $\Gal(K/\Q)$, so $\Frob_\lambda$ acts on $H^1_{sin}(K_\lambda,W_\p[p^m])^\pm$ as multiplication by $\pm1$. Then Lemma \ref{lemma5.2} shows that $\bigl(a_\ell\pm(\ell+1)\Frob_\lambda\bigr)\big/p^m$ are indeed well-defined endomorphisms of the $\cO_\p/p^m\cO_\p$-module $H^1_{sin}(K_\lambda, W_\p[p^m])$, and the last claim is obvious. \end{proof}

In the proof of the next result we use the isomorphism $\nu_\lambda$ defined in \eqref{fin-sin-isom-eq} (keep Remarks \ref{fin-notation-rem} and \ref{singular-rem} in mind for our notational conventions).

\begin{proposition} \label{finite-2-prop}
Suppose that $\ell\in\tilde{\mathcal S}_{p^m}$ and that $\bigl(a_\ell\pm(\ell+1)\bigr)\big/p^m$ are both $p$-adic units. Then $d(\ell)_\lambda\not=0$ in $H^1_{sin}(K_\lambda,W_\p[p^m])$ if and only if ${\bf D}_\kappa(y_{S,\p}{)}_\lambda\not=0$ in $H^1_f(K_\lambda,W_\p[p^m])$.  
\end{proposition}

\begin{proof} Applying the Key Formula in \cite[\S 9]{Nek} with $y$ a $1$-cocycle representing ${\bf D}_\kappa(y_{S\ell,\p})$, and using Proposition \ref{eigen-d-prop} plus the relations $\ell+1\equiv a_\ell\equiv0\pmod{\p^m}$ to simplify the right hand side, we get 
\[ \frac{(-1)^{k/2-1}\epsilon_\kappa a_\ell-(\ell+1)}{p^m}d(\ell)_{\lambda,sin}\equiv\frac{a_\ell-(\ell+1)\Frob_\lambda}{p^m}\Big(\nu_\lambda\bigl({\bf D}_\kappa(y_{S,\p}{)}_\lambda\bigr)\!\Big)\pmod{p^m}. \]
(Note the difference of sign with respect to \emph{loc. cit.}; the correction can be found in \cite[Proposition 5.16]{Nek-Dur}.) But $\frac{(-1)^{k/2-1}\epsilon_\kappa a_\ell-(\ell+1)}{p^m}\in(\mathcal O_\p/p^m\mathcal O_\p)^\times$ by assumption, and the proposition follows because $\frac{a_\ell-(\ell+1)\Frob_\lambda}{p^m}$ is invertible on $H^1_{sin}(K_\lambda,W_\p[p^m])$ by Lemma \ref{frob-invertible-lemma}. \end{proof}

Recall that the data $(\p^m,S,{\bf D}_{{\kappa}},\ell)$ satisfy Assumption \ref{ass}. As before, $L_m$ is the field $K(W_\p[p^m])$ and $S'$ is the conductor of ${\bf D}_\kappa$. Define
\begin{equation}\label{S'-Selmer}
H^1_{f,S'}(K,W_\p[p^m]):=\ker\Big(H^1_f(K,W_\p[p^m])\longrightarrow\bigoplus_{v\,|\,S'}H^1_f(K_v,W_\p[p^m])\Big). 
\end{equation}
The following result is the counterpart of \cite[Proposition 4.10]{Dar}.

\begin{proposition} \label{4.10prop}
The class $d(\ell)_\lambda$ lies in the kernel of 
\[ \phi_\lambda:H^1_{sin}(K_\lambda,W_\p[p^m])^{\epsilon_\kappa}\longrightarrow\bigl(H^1_{f,S'}(K,W_\p[p^m]{)}^*\bigr)^{\epsilon_\kappa} \]
for all $m\geq1$.
\end{proposition}

\begin{proof} To begin with, $d(\ell)_\lambda\in H^1_{sin}(K_\lambda,W_\p[p^m])^{\epsilon_\kappa}$ by Proposition \ref{eigen-d-prop}. Pick an element $s\in H^1_{f,S'}(K,W_\p[p^m])$, so that $s_\lambda\in H^1_f(K_\lambda,W_\p[p^m])$; we need to show that
\begin{equation} \label{claim-eq}
{\langle s_\lambda,d(\ell)_\lambda\rangle}_\lambda=0.
\end{equation} 
By \cite[Proposition 2.2, (2)]{Bes}, one has
\begin{equation} \label{sum-0-eq}
\sum_v{\langle s_v,d(\ell)_v\rangle}_v=0, 
\end{equation}
where the sum is taken over all (finite) places of $K$. Now observe that if $v\nmid S'\ell$ then ${\langle s_v,d(\ell)_v\rangle}_v=0$ by Proposition \ref{finite-1-prop}. On the other hand, if $v\,|\,S'$ then $s_v=0$ because $s\in H^1_{f,S'}(K,W_\p[p^m])$. Therefore \eqref{claim-eq} is an immediate consequence of \eqref{sum-0-eq}. \end{proof}

\subsection{Applications of \v{C}ebotarev's density theorem} 

Recall that $L_m$ is the composite of $K$ and the field $\bar\Q^{\ker(\bar\vartheta_{\p,m})}$ fixed by $\ker(\bar\vartheta_{\p,m})$. Similarly, define $L_{S,m}:=K_S(W_\p[p^m])$ to be the composite of $K_S$ and $\bar\Q^{\ker(\bar\vartheta_{\p,m})}$. 

We need some cohomological lemmas.

\begin{lemma} \label{disjoint-lem}
$H^i\bigl(\Gal(L_m/K),W_\p[p^m]\bigr)\simeq H^i\bigl(\Gal(L_{S,m}/K_S),W_\p[p^m]\bigr)$ for all $i\geq0$. 
\end{lemma}

\begin{proof} The fields $K$ and $\Q(W_\p[p^m])$ are linearly disjoint over $\Q$, and the same is true of $K_S$ and $\Q(W_\p[p^m])$. This holds because $\Q(W_\p[p^m])\cap K$ and $\Q(W_\p[p^m])\cap K_S$ are extensions of $\Q$ that are everywhere unramified, which is a consequence of the fact that $\Q(W_\p[p^m])/\Q$ is unramified outside $Np$ while $K_S/\Q$ is unramified outside $SD$ and $(pN,SD)=1$. It follows that the restriction maps induce canonical isomorphisms 
\[ \Gal(L_m/K)\simeq\Gal(\Q(W_\p[p^m])/\Q),\quad\Gal(L_{S,m}/K_S)\simeq\Gal(\Q(W_\p[p^m])/\Q). \] 
Therefore both groups appearing in the statement of the lemma are isomorphic to 
\[ H^1\bigl(\Gal(\Q(W_\p[p^m])/\Q),W_\p[p^m]\bigr), \] 
and this concludes the proof.\end{proof}

\begin{lemma} \label{vanishing-cohomology-lem}
$H^i\bigl(\Gal(L_m/K),W_\p[p^m]\bigr)=0$ for all $i\geq0$.
\end{lemma}

\begin{proof} This is \cite[Proposition 6.3, (2)]{Bes}. \end{proof}

\begin{lemma} \label{restriction-lem1}
The restriction map 
\[ H^1(K,W_\p[p^m])\longrightarrow H^1(K_S,W_\p[p^m]) \]
is injective.
\end{lemma}

\begin{proof} By the inflation-restriction exact sequence, the kernel of this map is 
\[ H^1\bigl(\Gal(K_S/K),W_\p[p^m](K_S)\bigr), \] 
which is trivial because $W_\p[p^m](K_S)=0$ by part (2) of Lemma \ref{lem-sol}. \end{proof}

\begin{lemma} \label{restriction-lem2}
The restriction map 
\[ H^1(K_S,W_\p[p^m])\longrightarrow H^1(L_{S,m},W_\p[p^m]) \]
is injective.
\end{lemma}

\begin{proof} The kernel of this map is $H^1\bigl(\Gal(L_{S,m}/K_S),W_\p[p^m]\bigr)$, which is trivial by a combination of Lemmas \ref{disjoint-lem} and \ref{vanishing-cohomology-lem}. \end{proof} 

\begin{lemma} \label{restriction-lem3}
The restriction map
\[ H^1(K,W_\p[p^m])\longrightarrow H^1(L_{S,m},W_\p[p^m]) \]
is injective.
\end{lemma}

\begin{proof} Combine Lemmas \ref{restriction-lem1} and \ref{restriction-lem2}. \end{proof}

Keep the notation of \eqref{r-selmer-eq}. Now we can prove

\begin{proposition}\label{prop-Cheb}
Suppose that
\begin{itemize}
\item ${\bf D}_\kappa(y_{S,\p})$ is not trivial in $H^1(K_S,W_\p[p^m])$;
\item $r_{\p,m}\bigl(H^1_{f,S'}(K,W_\p[p^m])^{\epsilon_\kappa}\bigr)\geq1$.
\end{itemize}
Then there exist infinitely many prime numbers $\ell$ such that
\begin{enumerate}
\item $\ell\nmid pNDS$ and $\Frob_\ell=\Frob_\infty\;\text{in}\;\Gal(L_{S,m}/\Q)$;
\item $\ell+1\pm a_\ell\not\equiv 0\pmod{\p^{m+1}}$; 
\item the image of  ${\bf D}_\kappa(y_{S,\p})$ in $H^1_f(K_\lambda,W_\p[p^m])$ is not zero, where $\lambda$ is the unique prime of $K$ above $\ell$;
\item the map of $\mathcal O_\p/p^m\mathcal O_\p$-modules 
\[H^1_{f,S'}(K,W_\p[p^m])^{\epsilon_\kappa}\longrightarrow H^1_f(K_\lambda,W_\p[p^m])^{\epsilon_\kappa}\] in surjective.
\end{enumerate}
\end{proposition}

\begin{proof}
By assumption, ${\bf D}_\kappa(y_{S,\p})\neq 0$ in $H^1(K_S,W_\p[p^m])$, hence Lemma \ref{restriction-lem2} implies that the same is true of its image in $H^1(L_{S,m},W_\p[p^m])$, denoted by the same symbol. Since $p$ is odd, $H^1(L_{S,m},W_\p[p^m])$ splits as the direct sum of its $\pm$-eigenspaces for the action of $c\in\Gal(K/\Q)$, so there is $\delta\in\{\pm\}$ such that the projection of ${\bf D}_\kappa(y_{S,\p})$ to the $\delta$-eigenspace of $H^1(L_{S,m},W_\p[p^m])$ is non-zero. Let us fix such a sign $\delta$ and write $d$ for the corresponding projection of ${\bf D}_\kappa(y_{S,\p})$.  

Let $s\in H^1_{f,S'}(K,W_\p[p^m])^{\epsilon_\kappa}$ be an element of exact order $p^m$, which exists by assumption. By Lemma \ref{restriction-lem3}, the image of $s$ in $H^1(L_{S,m},W_\p[p^m])$ has order $p^m$ as well. Since the relevant Galois action is trivial and $W_\p[p^m]$ is abelian, there is a canonical identification
\[ H^1(L_{S,m},W_\p[p^m])=\Hom\bigl(\Gal(L_{S,m}^{\rm ab}/L_{S,m}),W_\p[p^m]\bigr), \]
where $L_{S,m}^{\rm ab}$ is the maximal abelian extension of $L_{S,m}$. Denote by $\psi$ and $\varphi$ the homomorphisms corresponding to $d$ and to the image of $s$ in $H^1(L_{S,m},W_\p[p^m])$, respectively. 
 
Denote by $\tilde L_{S,m}$ the smallest extension of $L_{S,m}$ that is cut out by $\psi$ and $\varphi$ and is Galois over $\Q$. There is an isomorphism
\[ \Gal(\tilde L_{S,m}/K_{S})\simeq\Gal(\tilde L_{S,m}/L_{S,m})\rtimes \Gal(L_{S,m}/K_S). \] 
Complex conjugation $c$ acts on $\Gal(\tilde L_{S,m}/L_{S,m})$ by inner automorphisms, and we denote by $\Gal(\tilde L_{S,m}/L_{S,m})^+$ the subgroup of $\Gal(\tilde L_{S,m}/L_{S,m})$ that is fixed by $c$. Set 
\[ \Phi:=H^1\bigl(\Gal(\tilde L_{S,m}/L_{S,m}),W_\p[p^m]\bigr)=\Hom\bigl(\Gal(\tilde L_{S,m}/L_{S,m}),W_\p[p^m]\bigr). \]
By definition of $\tilde L_{S,m}$, the maps $\psi$ and $\varphi$ factor through $\Gal(\tilde L_{S,m}/L_{S,m})$ and so determine maps $\bar\psi$ and $\bar\varphi$ in $\Phi$. The group $\Gal(L_{S,m}/K_S)$ acts canonically on $\Phi$, and $\bar\psi$ and $\bar\varphi$ are fixed by this action as they are restrictions of classes in $H^1(K_S,W_\p[p^m])$ and $H^1(K,W_\p[p^m])$, respectively. There is also an action of $c$ on $\Phi$ and, since $s$ belongs to $H^1(K,W_\p[p^m])^{\epsilon_\kappa}$, the map $\bar\varphi$ belongs to $\Phi^{\epsilon_\kappa}$, while $\bar\psi$ belongs to $\Phi^\delta$ by construction. 

Now we claim that both $\bar\psi$ and $p^{m-1}\bar\varphi$ are non-zero on $\Gal(\tilde L_{S,m}/L_{S,m})^+$. To show this, let $\varrho$ denote either $\bar\psi$ or $p^{m-1}\bar\varphi$. If $\varrho=0$ on $\Gal(\tilde L_{S,m}/L_{S,m})^+$ then $\varrho$ maps $\Gal(\tilde L_{S,m}/L_{S,m})$ to one of the eigenspaces $W_\p[p^m]^\pm$; this is true because $\varrho$ factors through the $p$-primary part of $\Gal(\tilde L_{S,m}/L_{S,m})$, which splits as the sum of the two eigenspaces for the action of $c$, and $\varrho$ belongs to an eigenspace of $\Phi$. Since $\varrho$ is non-zero and fixed by $\Gal(L_{S,m}/K_S)$, it follows that ${\rm im}(\varrho)$ is a non-zero, proper submodule of $W_\p[p^m]$ that is stable under the action of $\Gal(L_{S,m}/K_S)\simeq\Gal(\Q(W_\p[p^m])/\Q)$. Multiplying ${\rm im}(\varrho)$ by a suitable power of $p$, we obtain a non-zero, proper submodule of $W_\p[p]$ that is stable under $\Gal(\Q(W_\p[p^m])/\Q)$, and this contradicts the irreducibility of 
$\bar\vartheta_\p$ (Lemma \ref{lem-irr}). We conclude that both $p^{m-1}\bar\varphi$ and $\bar\psi$ are necessarily non-zero on $\Gal(\tilde L_{S,m}/L_{S,m})^+$. 

It follows that we can find $g\in\Gal(\tilde L_{S,m}/L_{S,m})^+$ such that $\bar\psi(g)\neq0$ and $\bar\varphi(g)$ has exact order $p^m$. Let $\ell$ be a prime number 
unramified in $\tilde L_{S,m}/\Q$ such that 
\begin{equation} \label{ell-eq}
\ell\nmid NDSp,\qquad\Frob_\ell=\Frob_\infty g.
\end{equation} 
Here $\Frob_\infty g$ denotes the conjugacy class of $cg$ in $\Gal(\tilde L_{S,m}/\Q)$. By \v{C}ebotarev's density theorem, the set of primes satisfying \eqref{ell-eq} is infinite. 

Clearly, (1) is satisfied by any $\ell$ as in \eqref{ell-eq}. In particular, $\ell$ is inert in $K$ and we denote by $\lambda$ the unique prime of $K$ above $\ell$, which splits completely in $L_{S,m}$ (\cite[Lemma 6.7]{Bes}). Choose a prime $\tilde\lambda_{S,m}$ of $\tilde L_{S,m}$ above $\lambda$ such that $\Frob_{\tilde\lambda_{S,m}/\ell}=cg$ and let $\lambda_{S,m}$ be the prime of $L_{S.m}$ below $\tilde\lambda_{S,m}$; the completion $L_{\lambda_{S,m}}$ of $L_{S,m}$ at $\lambda_{S,m}$ is then equal to $K_\lambda$.

Now we show that every prime $\ell$ satisfying \eqref{ell-eq} satisfies also (3) and (4) in the statement of the proposition. If $\varrho=\bar\psi$ or $\varrho=p^{m-1}\bar\varphi$ then, since $g^c=g$, one has 
\begin{equation} \label{varrho-frob-eq} 
\varrho\bigl(\Frob_{\tilde\lambda_{S,m}/\lambda_{S,m}}\bigr)=\varrho\bigl(\Frob_{\tilde\lambda_{S,m}/\ell}^2\bigr)=\varrho(g^cg)=\varrho(g^2)\neq0. 
\end{equation}
Therefore the restriction of $\varrho$ to $\Gal\bigl(\tilde L_{\tilde\lambda_{S,m}}/L_{\lambda_{S,m}}\bigr)$ is non-zero and hence, since $L_{\lambda_{S,m}}=K_\lambda$, taking $\varrho=\bar\psi$ gives (3). As for (4), note that, by Lemma \ref{dimension}, it suffices to find an element of $H^1_{f,S'}(K_\lambda,W_\p[p^m])^{\epsilon_\kappa}$ whose image in $H^1_f(K_\lambda,W_\p[p^m])^{\epsilon_\kappa}$ has exact order $p^m$. But it follows from \eqref{varrho-frob-eq} with $\varrho=p^{m-1}\bar\varphi$ that the order of the image of $s$ in $H^1_f(K_\lambda,W_\p[p^m])^{\epsilon_\kappa}$ is $p^m$, and we are done.

Finally, we show that one can choose infinitely many primes $\ell$ as in \eqref{ell-eq} such that (2) is true. Fix a prime $\ell'$ satisfying \eqref{ell-eq} but not (2), so that $\ell'+1\equiv\epsilon a_{\ell'}\pmod{\p^{m+1}}$ for a suitable $\epsilon\in\{\pm1\}$. It is known that $\tr(F_{\ell'}\,|\,A_\p)=a_{\ell'}/\ell'^{k/2-1}$ and $\det(F_{\ell'}\,|\,A_\p)=\ell'$. Take any $\alpha\in \Z_p^\times$ such that $\alpha\equiv1\pmod{p^m}$ and set $M:=\smallmat\alpha00\alpha$. By \cite[Lemma 6.2]{Bes}, the matrix $M$ lies in the image of the representation $\vartheta_\p$ of $G_\Q$ on $A_\p$, hence there is $\sigma_\alpha\in G_\Q$ having $M$ as its image. Then
\[ \tr(F_{\ell'}\sigma_\alpha\,|\,A_\p)=\alpha a_{\ell'}\big/\ell'^{k/2-1},\qquad\det(F_{\ell'}\sigma_\alpha\,|\,A_\p)=\alpha^2\ell'. \]
Let $\ell$ be a prime such that $\ell\nmid NDSp$ and $\Frob_\ell= \Frob_{\ell'}{\sigma_\alpha|}_{\tilde L_{S,m+1}}$ in $\Gal(\tilde L_{S,m+1}/\Q)$, where we denote by $\Frob_{\ell'}{\sigma_\alpha|}_{\tilde L_{S,m+1}}$ the conjugacy class of $f\cdot{\sigma_\alpha|}_{\tilde L_{S,m+1}}$ for any choice of $f\in\Frob_{\ell'}$. Again, \v{C}ebotarev's density theorem guarantees that there exist infinitely many such $\ell$. Then
\[ a_\ell\big/\ell^{k/2-1}\equiv\alpha a_{\ell'}\big/\ell'^{k/2-1}\pmod{\p^{m+1}},\qquad\ell\equiv\alpha^2\ell'\pmod{\p^{m+1}}, \] 
and one deduces that there exists an $\alpha$ as above such that $\ell+1\pm a_\ell\not\equiv0\pmod{\p^{m+1}}$. This shows that $\ell$ satisfies (2). But the image of $\Frob_\ell$ in $\Gal(\tilde L_{S,m}/\Q)$ is equal to that of $\Frob_{\ell'}$, and so $\ell$ satisfies \eqref{ell-eq} too. \end{proof}

\subsection{Divisibility properties of Heegner cycles} \label{main-subsec} 

The arguments in this subsection follow those in \cite[\S 5.1]{Dar}. As before, $\ell$ belongs to $\mathcal S_{p^m}$ and $\lambda$ is the unique prime of $K$ above $\ell$. Moreover, recall the shorthand $\mathbb F_\p=\mathcal O_\p/p\mathcal O_\p$ and for any $\mathcal O_\p/p^m\mathcal O_\p$-module $M$ set 
\[ r_\p(M):=\dim_{\F_\p}\bigl(M\otimes_{\cO_\p/p^m\cO_\p}\!\F_\p\bigr). \] 
To use uniform notations, put also $r_\p:=r_{\p,1}$. 

\begin{lemma} \label{dimension-2}
$ r_\p\bigl(H^1_f(K_\lambda,W_\p[p^m])^\pm\bigr)\leq1$.
\end{lemma}

\begin{proof} To ease notations, in this proof we use the symbol $\otimes$ to denote tensorization over $\cO_\p/p^m\cO_p$. With this convention in mind, note that for any $\cO_\p/p^m\cO_\p$-module $M$ equipped with an action of $\Gal(K/\Q)$ there are injections
\begin{equation} \label{M-eigen-eq}
M^\pm\otimes\F_\p\;\longmono\;(M\otimes\F_\p)^\pm.
\end{equation}
If $\widehat\Z$ is the profinite completion of $\Z$ then $\Gal(K_\lambda^{\rm ur}/K_\lambda)\simeq\widehat\Z$, hence well-known results in group cohomology (see, e.g., \cite[Ch. XIII, Proposition 1]{Se}) show that there is a short exact sequence
\begin{equation} \label{finite-frob-eq}
0\longrightarrow (\Frob_\lambda-1)W_\p[p^m]\longrightarrow W_\p[p^m]\longrightarrow H^1_f(K_\lambda,W_\p[p^m])\longrightarrow0. 
\end{equation}
Tensoring \eqref{finite-frob-eq} with $\F_\p$ produces an exact sequence
\begin{equation} \label{finite-frob-eq2}
(\Frob_\lambda-1)W_\p[p^m]\otimes\F_\p\overset\iota\longrightarrow W_\p[p^m]\otimes\F_\p\longrightarrow H^1_f(K_\lambda,W_\p[p^m])\otimes\F_\p\longrightarrow0. 
\end{equation}
By \cite[Proposition 6.3, (4)]{Bes}, $W_\p[p^m]^\pm$ is free of rank $1$ over $\cO_\p/p^m\cO_\p$, and then \eqref{M-eigen-eq} with $M=W_\p[p^m]$ gives
\begin{equation} \label{W-eigen-dim-eq}
\dim_{\F_\p}\bigl((W_\p[p^m]\otimes\F_\p)^\pm\bigr)=1. 
\end{equation}
If ${\rm im}(\iota)=0$ then \eqref{finite-frob-eq2} induces isomorphisms 
\begin{equation} \label{W-finite-eigen-eq}
\bigl(W_\p[p^m]\otimes\F_\p\bigr)^\pm\simeq\bigl(H^1_f(K_\lambda,W_\p[p^m])\otimes\F_\p\bigr)^\pm,
\end{equation}
and the inequalities $r_\p\bigl(H^1_f(K_\lambda,W_\p[p^m])^\pm\bigr)\leq1$ follow by combining \eqref{W-eigen-dim-eq}, \eqref{W-finite-eigen-eq} and \eqref{M-eigen-eq} with $M=H^1_f(K_\lambda,W_\p[p^m])$. Finally, $W_\p[p^m]\otimes\F_\p$ has dimension $2$ over $\F_\p$, so if ${\rm im}(\iota)\not=0$ then \eqref{finite-frob-eq2} implies that $r_\p\bigl(H^1_f(K_\lambda,W_\p[p^m])\bigr)\leq1$ and, \emph{a fortiori}, $r_\p\bigl(H^1_f(K_\lambda,W_\p[p^m])^\pm\bigr)\leq1$. \end{proof}
          
\begin{remark} 
If $\ell\in\tilde{\mathcal S}_{p^m}$ then Lemma \ref{dimension} shows that equality holds in Lemma \ref{dimension-2}.
\end{remark}          
          
To simplify our notation, for every integer $S'>1$ define 
\begin{equation} \label{A(S')-eq}
A(S'):=\bigoplus_{\lambda\mid S'}H^1_f(K_\lambda,W_\p[p^m]). 
\end{equation}
Of course, the module $A(S')$ depends on $m$, but no confusion is likely to arise.

\begin{lemma} \label{aux-lemma}
If $\ell\in\tilde{\mathcal S}_{p^m}$ then
\[ r_{\p,m}\bigl(H^1_{f,S'}(K,W_\p[p^m])^\pm\bigr)\leq r_{\p,m}\bigl(H^1_{f,S'\ell}(K,W_\p[p^m])^\pm\bigr)+r_\p\bigl(A(S'\ell)^\pm\bigr)-r_\p\bigl(A(S')^\pm\bigr). \]
\end{lemma}

\begin{proof} There is an exact sequence 
\[ 0\longrightarrow H^1_{f,S'\ell}(K,W_\p[p^m])^\pm\longrightarrow H^1_{f,S'}(K,W_\p[p^m])^\pm\longrightarrow H^1_f(K_\lambda,W_\p[p^m])^\pm \] 
where $\lambda$ is the prime of $K$ above $\ell$. Combining part (3) of Lemma \ref{basic-lemma} and the obvious inequality 
\[ r_{\p,m}\bigl(H^1_f(K_\lambda,W_\p[p^m])^\pm\bigr)\leq r_\p\bigl(H^1_f(K_\lambda,W_\p[p^m])^\pm\bigr) \]
we find 
\[ r_{\p,m}\bigl(H^1_{f,S'}(K,W_\p[p^m])^\pm\bigr)\leq r_{\p,m}\bigl(H^1_{f,S'\ell}(K,W_\p[p^m])^\pm\bigr)+r_\p\bigl(H^1_f(K_\lambda,W_\p[p^m])^\pm\bigr). \]
Applying Lemma \ref{dimension-2} to the above inequality yields  
\begin{equation} \label{eq34}
r_{\p,m}\bigl(H^1_{f,S'}(K,W_\p[p^m])^\pm\bigr)\leq r_{\p,m}\bigl(H^1_{f,S'\ell}(K,W_\p[p^m])^\pm\bigr)+1.
\end{equation}
Now $\ell$ belongs to $\tilde{\mathcal S}_{p^m}$, so by Lemma \ref{dimension} one has  
\[ r_\p\bigl(H^1_f(K_\lambda,W_\p[p^m])^\pm\bigr)=1, \] 
and we deduce that 
\[ r_\p\bigl(A(S'\ell)^\pm\bigr)=r_\p\bigl(A(S')^\pm\bigr)+1. \] 
Hence inequality \eqref{eq34} becomes 
\[ r_{\p,m}\bigl(H^1_{f,S'}(K,W_\p[p^m])^\pm\bigr)\leq r_{\p,m}\bigl(H^1_{f,S'\ell}(K,W_\p[p^m])^\pm\bigr)+r_\p\bigl(A(S'\ell)^\pm\bigr)-r_\p\bigl(A(S')^\pm\bigr), \] 
as was to be shown. \end{proof}

\begin{proposition} \label{T5.20}
Let ${\bf D}_\kappa$ be a derivative of support $S$ and conductor $S'$. If 
\[ \ord({\bf D}_{\kappa})< r_{\p,m}\bigl(H^1_{f,S'}(K,W_\p[p^m])^{\epsilon_\kappa}\bigr)+r_\p\bigl(A(S')^{\epsilon_\kappa}\bigr) \] 
then ${\bf D}_{{\kappa}}(y_{S,\p})\equiv0\pmod{p^m}$. 
\end{proposition}

\begin{proof} Define the \emph{weight} of ${\bf D}_\kappa$ to be 
\[ {\rm wt}({\bf D}_\kappa):=\ord({\bf D}_\kappa)-\#\bigl\{\text{$\ell$ prime number}\;\big|\;\text{$\ell\,|\,S$ and $\ell\in\tilde{\mathcal S}_{p^m}$}\bigr\}. \]
The prove the proposition we proceed by induction on ${\rm wt}({\bf D}_\kappa)$.

First of all, observe that if ${\rm wt}({\bf D}_\kappa)<0$ then the result is true. Indeed, in this case ${\bf D}_\kappa$ contains at least one factor of the form ${\bf D}_\ell^0$ for some prime $\ell\in\tilde{\mathcal S}_{p^m}$. By part (1) of Proposition \ref{complex-conj}, and the relation \eqref{res-cores-norm} between restriction, corestriction and Galois trace, we have 
\[ {\bf D}_{\ell}^0(y_{T\ell,\p})=\res_{K_{T\ell}/K_T}(y_{T,\p})\cdot (a_\ell/\ell^{k/2-1})\equiv 0\pmod{p^m}, \] 
where the congruence holds because $\ell\in\tilde{\mathcal S}_{\p^m}$ (here $\res_{K_{T\ell}/K_T}$ denotes the restriction map in cohomology from $H^1(K_T,A_\p)$ to $H^1(K_{T\ell},A_\p)$). Then the result follows (without assuming any condition on the order of ${\bf D}_\kappa$).

Now set $k:={\rm wt}({\bf D}_\kappa)$ and assume by induction that the theorem is true for all derivatives ${\bf D}_{\kappa'}$ such that ${\rm wt}({\bf D}_{\kappa'})<k$. We argue by contradiction, supposing that 
\begin{equation} \label{non-van-eq}
{\bf D}_{{\kappa}}(y_{S,\p})\not\equiv 0\pmod{p^m}.
\end{equation}
We first show that the inequality in the statement of the proposition plus \eqref{non-van-eq} imply that 
\begin{equation} \label{eq30}
r_{\p,m}\big(H^1_{f,S'}(K,W_\p[p^m])^{\epsilon_\kappa}\big)\geq 1.\end{equation}
In fact, if this were not the case then there would be an inequality  
\[ \ord({\bf D}_\kappa)<r_\p\bigl(A(S')\bigr). \]
By Lemma \ref{dimension-2}, the right hand side of this inequality is less than or equal to the number of primes dividing $S'$; but each of them contributes at least for $1$ unity in the sum defining $\ord({\bf D}_\kappa)$, so the above inequality does not occur and we conclude that \eqref{eq30} holds.  

Equations \eqref{non-van-eq} and \eqref{eq30} show that the assumptions in Proposition \ref{prop-Cheb} are fulfilled and therefore, with the usual notation, one can find a prime number $\ell$ such that
\begin{itemize}
\item $\ell\nmid pNDS$ and $\Frob_\ell=\Frob_\infty\;\text{in}\;\Gal(L_{S,m}/\Q)$;
\item $\p^{m+1}\nmid (\ell+1)\pm a_\ell$; 
\item the image of  ${\bf D}_\kappa(y_{S,\p})$ in $H^1_f(K_\lambda,W_\p[p^m])$ is not zero;
\item the map of $\mathcal O_\p/p^m\mathcal O_\p$-modules 
\begin{equation} \label{S'-eigenspaces-eq}
H^1_{f,S'}(K_\lambda,W_\p[p^m])^{\epsilon_\kappa}\longrightarrow H^1_f(K_\lambda,W_\p[p^m])^{\epsilon_\kappa}
\end{equation} 
is surjective.
\end{itemize} 
Dualizing the map in \eqref{S'-eigenspaces-eq} and using \eqref{eigenspaces-pm-isom-eq} and \eqref{isom-pontryagin-eq}, we see that the map
\begin{equation} \label{eq31-bis}
\phi_\lambda:H^1_{sin}(K_\lambda,W_\p[p^m]{)}^{\epsilon_\kappa}\longrightarrow\bigl(H^1_{f,S'}(K,W_\p[p^m]{)}^{\!*}\bigr)^{\epsilon_\kappa}
\end{equation} 
is injective. Now we want to show that the derivative ${\bf D}_\kappa{\bf D}_\ell^1$ satisfies Assumption \ref{ass}. Fix ${\bf D}_{{\kappa}'}$ strictly less than ${\bf D}_{{\kappa}}{\bf D}_\ell^1$. Then  
\[ \ord({\bf D}_{\kappa'})<\ord({\bf D}_\kappa{\bf D}_\ell^1)=\ord({\bf D}_\kappa)+1, \] 
hence 
\begin{equation} \label{eq31}
\ord({\bf D}_{\kappa'})\leq \ord({\bf D}_\kappa).
\end{equation}
By Lemma \ref{aux-lemma}, one has 
\[ r_{\p,m}\bigl(H^1_{f,S'}(K,W_\p[p^m])^{\epsilon_\kappa}\bigr)\leq r_{\p,m}\bigl(H^1_{f,S'\ell}(K,W_\p[p^m])^{\epsilon_\kappa}\bigr)+r_\p\bigl(A(S'\ell)^{\epsilon_\kappa}\bigr)-r_\p\bigl(A(S')^{\epsilon_\kappa}\bigr). \] 
Combining this inequality with the one in the statement of the proposition we find 
\[ \ord({\bf D}_{{\kappa}})< r_{\p,m}\bigl(H^1_{f,S'\ell}(K,W_\p[p^m])^{\epsilon_\kappa}\bigr)+r_\p\bigl(A(S'\ell)^{\epsilon_\kappa}\bigr) \] 
and therefore, applying \eqref{eq31}, we get 
\begin{equation} \label{eq32} 
\ord({\bf D}_{{\kappa'}})< r_{\p,m}\bigl(H^1_{f,S'\ell}(K,W_\p[p^m])^{\epsilon_\kappa}\bigr)+r_\p\bigl(A(S'\ell)^{\epsilon_\kappa}\bigr).
\end{equation}
Furthermore, since the support of ${\bf D}_{\kappa'}$ is divisible by an extra prime $\ell\in\tilde{\mathcal S}_{p^m}$, we see that 
\begin{equation} \label{eq33}
{\rm wt}({\bf D}_{\kappa'})<{\rm wt}({\bf D}_\kappa).
\end{equation}
Equations \eqref{eq32} and \eqref{eq33} show that ${\bf D}_{\kappa'}$ satisfies the induction hypothesis, and we conclude that ${\bf D}_{\kappa'}(y_{S\ell,\p})\equiv0\pmod{p^m} $. This shows that Assumption \ref{ass} is satisfied in our setting, hence we may apply the construction of \S \ref{heegner-classes-subsec} and obtain a class $d(\ell)\in H^1(K,W_\p[p^m])$. Since the image of  ${\bf D}_\kappa(y_{S,\p})$ in $H^1_f(K_\lambda,W_\p[p^m])$ is not zero, it follows from Proposition \ref{finite-2-prop} (which we can apply because $\p^{m+1}\nmid\ell+1\pm a_\ell$) that the image of $d(\ell)_\lambda$ in $H^1_{sin}(K_\lambda,W_\p[p^m])$ is non-zero too. Therefore, since $d(\ell)$ belongs to the $\epsilon_\kappa$-eigenspace for $c$ thanks to Proposition \ref{eigen-d-prop}, Proposition \ref{4.10prop} ensures that the map 
\[ \phi_\lambda:H^1_{sin}(K_\lambda,W_\p[p^m])^{\epsilon_\kappa}\longrightarrow\bigl(H^1_{f,S'}(K,W_\p[p^m])^*\bigr) ^{\epsilon_\kappa} \] 
is not injective. But this contradicts \eqref{eq31-bis}, and the proposition is proved. \end{proof}
  
Now we keep notations and assumptions as in Proposition \ref{T5.20} and prove two corollaries. 

\begin{corollary} \label{coro5.21}
If 
\[ \ord({\bf D}_\kappa)< r_{\p,m}\bigl(H^1_{f,S'}(K,W_\p[p^m])^{-\epsilon_\kappa}\bigr)+r_\p\bigl(A(S')^{-\epsilon_\kappa}\bigr)-1 \] 
and $\ord({\bf D}_\kappa)<p$ then ${\bf D}_\kappa(y_{S,\p})\equiv0\pmod{p^m} $.
\end{corollary}

\begin{proof} Suppose ${\bf D}_{\kappa}(y_{S,\p})\not\equiv0\pmod{p^m} $ and pick a prime $\ell$ such that $\Frob_\ell=\Frob_\infty$ in $\Gal(L_{S,m}/\Q)$ and the image of  ${\bf D}_\kappa(y_{S,\p})$ in $H^1_f(K_\lambda,W_\p[p^m])$ is not zero (that such a choice is possible can be checked along the same lines as in the proof of Proposition \ref{prop-Cheb}, and the arguments are actually simpler). 

Now we show that 
\begin{equation} \label{eq39}
\text{${\bf D}_\kappa{\bf D}_\ell^1(y_{S\ell,\p})$ is not zero in $H^1(K,W_\p[p^m])$.} 
\end{equation}
If there is a derivative ${\bf D}_{\kappa'}$ strictly less than ${\bf D}_\kappa{\bf D}_\ell^1$ such that ${\bf D}_{\kappa'}(y_{S\ell,\p})$ is not zero in $H^1(K,W_\p[p^m])$, using formula \eqref{sigma-ell-D-eq} recursively one easily shows that \eqref{eq39} holds (use here the fact that $\ord({\bf D}_\kappa)<p$). On the contrary, if for all derivatives ${\bf D}_{\kappa'}$ strictly less than ${\bf D}_\kappa{\bf D}_\ell^1$ we have ${\bf D}_{\kappa'}(y_{S\ell,\p})=0$ in $H^1(K,W_\p[p^m])$ then one can construct the class $d(\ell)$ which, by Proposition \ref{finite-2-prop}, is not locally trivial at $\lambda$. Hence, \emph{a fortiori}, $d(\ell)$ is not globally trivial, and therefore also $P(\ell)={\bf D}_\kappa{\bf D}_\ell^1(y_{S\ell,\p})$ is not trivial. 

At this point we make use of our assumptions. Since 
\begin{equation} \label{eq38}
\ord({\bf D}_{\kappa}{\bf D}_\ell^1)=\ord({\bf D}_\kappa)+1,
\end{equation} 
it follows that 
\[ \ord({\bf D}_{\kappa}{\bf D}_\ell^1)< r_{\p,m}\bigl(H^1_{f,S'}(K,W_\p[p^m])^{-\epsilon_\kappa}\bigr)+r_\p\bigl(A(S')^{-\epsilon_\kappa}\bigr). \]
By Lemma \ref{aux-lemma}, the right hand side of the above inequality is less than or equal to 
\[ r_{\p,m}\bigl(H^1_{f,S'\ell}(K,W_\p[p^m])^{-\epsilon_\kappa}\bigr)+r_\p\bigl(A(S'\ell)^{-\epsilon_\kappa}\bigr) \]
and therefore we obtain the inequality 
\[ \ord({\bf D}_{\kappa}{\bf D}_\ell^1)<r_{\p,m}\bigl(H^1_{f,S'\ell}(K,W_\p[p^m])^{-\epsilon_\kappa}\bigr)+r_\p\bigl(A(S'\ell)^{-\epsilon_\kappa}\bigr). \] 
By \eqref{eq38}, we have $(-1)^{\ord({\bf D}_\kappa{\bf D}_\ell^1)}=-\epsilon_\kappa$. Therefore we can apply Proposition \ref{T5.20}, which shows that ${\bf D}_\kappa{\bf D}_\ell^1(y_{S\ell,\p})\equiv0\pmod{p^m}$. In light of \eqref{eq39}, this is a contradiction. \end{proof} 

\begin{corollary} \label{coro5.22} 
If one of the two conditions
\begin{enumerate}
\item $\ord({\bf D}_{\kappa})< r_{\p,m}\bigl(H^1_f(K,W_\p[p^m])^{\epsilon_\kappa}\bigr)$, 
\item $\ord({\bf D}_{\kappa})< r_{\p,m}\bigl(H^1_f(K,W_\p[p^m])^{-\epsilon_\kappa}\bigr)-1$ and $\ord({\bf D}_\kappa)<p$ 
\end{enumerate}
holds then ${\bf D}_{\kappa}(y_{S,\p})\equiv 0\pmod{p^m}$.
\end{corollary}

\begin{proof} Part (3) of Lemma \ref{basic-lemma} implies that 
\[ r_{\p,m}\bigl(H^1_f(K,W_\p[p^m])^{\epsilon_\kappa}\bigr)\leq r_{\p,m}\bigl(H^1_{f,S'}(K,W_\p[p^m])^{\epsilon_\kappa}\bigr)+r_\p\bigl(A(S')^{\epsilon_\kappa}\bigr). \] 
The corollary follows from  Proposition \ref{T5.20} if condition (1) holds and from Corollary \ref{coro5.21} if condition (2) holds. \end{proof}

We are now in a position to state and prove the main result of this section. 

\begin{theorem} \label{div-thm} 
Let $S$ be a square-free product of primes in $\mathcal S_{p^m}$. If $\ord({\bf D}_{{\kappa}})<\min\{r_{\p,m},p\} $ then ${\bf D}_{{\kappa}}(y_{S,\p})\equiv 0\pmod{p^m}$.  
\end{theorem}

\begin{proof} Since $\ord({\bf D}_{{\kappa}})<r_{\p,m} $ and 
\[ r_{\p,m}=r_{\p,m}\bigl(H^1_f(K,W_\p[p^m])^{\epsilon_\kappa}\bigr)+ r_{\p,m}\bigl(H^1_f(K,W_\p[p^m])^{-\epsilon_\kappa}\bigr), \]
at least one of the conditions in Corollary \ref{coro5.22} is satisfied (in (2) we also need the condition $\ord({\bf D}_\kappa)<p$, which is not needed for (1)), and we are done. \end{proof}

\section{Theta elements and refined Beilinson--Bloch conjecture} \label{secMR} 

In this section we prove our main result on the order of vanishing of certain combinations of Heegner cycles. 

\subsection{Theta elements and arithmetic $L$-functions} \label{sec-theta}

For any square-free product  $T$ of prime numbers belonging to the set $\mathcal S$ defined in \eqref{S-set-eq} consider the resolvent element
\[ \theta_{T,\p}:=\sum_{\sigma\in G_T}\sigma(y_{T,\p})\otimes\sigma\in\BK_\p(K_T)\otimes_{\cO_\p}\!\cO_\p[G_T]. \] 
Our main result relates these elements to the dimension of $X_\p(K)$ over $F_\p$. 

We also need to introduce suitable variants and combinations of the above elements. To begin with, we trace them down to $K$ as follows. As in \S \ref{heegner-classes-subsec}, fix any lift ${\bf N}_T\in\Z[\Gamma_T]$ of the norm ${\bf N}=\sum_{\sigma\in\Gamma_1}\sigma$; in other words, for every $\sigma\in\Gamma_1$ choose $\sigma'\in\Gamma_T$ such that ${\sigma'|}_{K_1}=\sigma$. Define 
\begin{equation} \label{zeta-eq}
\zeta_{T,\p}:={\bf N}_T(\theta_{T,\p})= \sum_{\sigma\in G_T}\sigma{\bf N}_T(y_{T,\p})\otimes\sigma\in\BK_\p(K_T)\otimes_{\cO_\p}\!\cO_\p[G_T]. 
\end{equation}
Note that these elements depend on the choice of ${\bf N}_T$, but for simplicity we shall drop this dependence from the notation. 

Let $x\mapsto x^*$ denote the involution of $\mathcal O_\p[G_T]$ induced by the map $\sigma\mapsto\sigma^{-1}$ on $G_T$ and denote by $\zeta^*_{T,\p}$ the element obtained by applying to $\zeta_{T,\p}$ the map induced by this involution. 

Fix a square-free product  $S$ of primes belonging to $\mathcal S$. As before, fix a lift ${\bf N}_S$ of ${\bf N}$ to $\Z[\Gamma_S]$. By projection, this gives lifts ${\bf N}_T$ for all $T\,|\,S$ that may be used to define $\zeta_{T,\p}$ and $\zeta_{T,\p}^*$ as in \eqref{zeta-eq}. Since the extension $K_S/\Q$ is generalized dihedral and hence solvable, part (1) of Lemma \ref{lem-sol} ensures that $A_\p(K_S)=0$, so for every $T\,|\,S$ the inflation-restriction exact sequence yields an injection $\BK_\p(K_T)\hookrightarrow\BK_\p(K_S)$. On the other hand, the natural inclusion $G_T\subset G_S$ (see \S \ref{secHC}) induces an injection $\cO_\p[G_T]\hookrightarrow\cO_\p[G_S]$ of (free) $\cO_\p$-modules, and therefore we obtain an injection
\begin{equation} \label{cohom-T_S-eq}
\BK_\p(K_T)\otimes_{\cO_\p}\!\cO_\p[G_T]\;\longmono\;\BK_\p(K_S)\otimes_{\cO_\p}\!\cO_\p[G_S]
\end{equation}
of $\cO_\p$-modules. Furthermore, the canonical inclusion $G_S\subset\Gamma_S$ induces an 
injection 
\begin{equation} \label{cohom-T_S-eq-I}
\BK_\p(K_S)\otimes_{\cO_\p}\!\cO_\p[G_S]\;\longmono\;\BK_\p(K_S)\otimes_{\cO_\p}\!\cO_\p[\Gamma_S]
\end{equation} 
of $\mathcal O_\p$-modules. The composition of \eqref{cohom-T_S-eq} and \eqref{cohom-T_S-eq-I} allows us to view $\zeta_{T,\p}$ and $\zeta^*_{T,\p}$ as elements of $\BK_\p(K_S)\otimes_{\cO_\p}\!\cO_\p[\Gamma_S]$, which from here on we shall do without any further warning. 
 
For $S$ fixed as above and every $T\,|\,S$ set 
\begin{equation} \label{def-multipl}
a_T:=\mu(T)\sum_{\sigma\in\Gal(K_S/K_T)}\sigma,\qquad a_T^*:=\chi_K(T)a_T
\end{equation}
where $\mu$ is the M\"obius function and $\chi_K$ is the quadratic character attached to $K$. Define the \emph{arithmetic $L$-function} attached to $S$ and $\p$ as
\begin{equation} \label{def-L}
{\mathcal L}_{S,\p}:=\bigg(\sum_{T\mid S}a_T\zeta_{T,\p}\bigg)\!\otimes\!\bigg(\sum_{T\mid S}a_T^*\zeta_{T,\p}^*\bigg)\in\BK_\p(K_S)^{\otimes2}\otimes_{\cO_\p}\!\cO_\p[\Gamma_S].
\end{equation} 
Here we are using the canonical identification 
\[ \BK_\p(K_S)^{\otimes2}\otimes_{\cO_\p}\!\cO_\p[\Gamma_S]=\bigl(\BK_\p(K_S)\otimes_{\cO_\p}\!\cO_\p[\Gamma_S]\bigr)\otimes_{\cO_\p[\Gamma_S]}\!\bigl(\BK_\p(K_S)\otimes_{\cO_\p}\!\cO_\p[\Gamma_S]\bigr), \]
the superscript ``$\otimes2$'' denoting tensorization over $\cO_\p$. Note that if $T\,|\,S$ and 
\[ \mu_{S,T}:\BK_\p(K_S)^{\otimes 2}\otimes_{\cO_\p}\!\cO_\p[\Gamma_S]\longrightarrow\BK_\p(K_S)^{\otimes 2}\otimes_{\cO_\p}\!\cO_\p[\Gamma_T] \] 
is the map induced by the canonical projection $\Gamma_S\twoheadrightarrow\Gamma_T$ then 
\begin{equation}\label{compatibility}
\mu_{S,T}(\mathcal L_{S,\p})=\mathcal L_{T,\p}\cdot\prod_{\ell\mid(S/T)}(1+\ell-a_\ell/\ell^{k/2-1})\cdot(1+\ell+a_\ell/\ell^{k/2-1}).
\end{equation}

\begin{remark} \label{rem4.1} 
One could define an element $\mathcal L_{S,\p}$ as in \eqref{def-L} by replacing the coefficients $a_T$ and $a_T^*$ with any choice of $b_T$ and $b_T'$ in $\mathcal O_\p[\Gamma_S]$, obtaining compatibility relations similar to \eqref{compatibility}. Our preference is motivated by the existence, still only conjectured, of a Mazur--Tate type regulator $\Reg(S)$ enjoying properties analogous to those of the regulator defined in \cite{MT0} and \cite{MT} and appearing in \cite{Dar}. This regulator $\Reg(S)$ should be used to express the leading value of $\mathcal L_{S,\p}$ for this specific choice of $a_T$ and $a_T^*$ (see Section \ref{sec-regulators} for more details). However, it is reasonable to expect alternative choices of coefficients $b_T$ and $b_T'$ to be related to other types of regulators having formal properties different from those of Mazur--Tate regulators. Finally, observe that the results for $\mathcal L_{S,\p}$ proved in this paper still hold for any choice of $b_
 T$ and $b_T'$: see Remarks \ref{remark-general-choice} and \ref{remark-general-choice-2} below.  
\end{remark} 

\subsection{Results on the order of vanishing} 

Recall that $I_{G_S}$ and $I_{\Gamma_S}$ are the augmentation ideals of $\cO_\p[G_S]$ and $\cO_\p[\Gamma_S]$, respectively. The powers of $I_{G_S}$ define a decreasing filtration
\begin{equation} \label{filtration-eq}
\cO_\p[G_S]=I_{G_S}^0\supset I_{G_S}^1\supset I_{G_S}^2\supset\dots\supset I_{G_s}^n\supset\cdots
\end{equation}
on $\cO_\p[G_S]$. On the other hand, since the $\cO_\p$-module $\BK_\p(K_S)$ is not in general torsion-free, we cannot expect tensorization of the sequence \eqref{filtration-eq} by $\BK_\p(K_S)$ over $\cO_\p$ to yield a filtration on $\BK_\p(K_S)\otimes_{\cO_\p}\!\cO_\p[G_S]$. In light of this, when we write that an element $\theta$ of $\BK_\p(K_S)\otimes_{\cO_\p}\!\cO_\p[G_S]$ belongs to $\BK_\p(K_S)\otimes_{\cO_\p}\!I_{G_S}^r$ we really mean that $\theta$ belongs to the natural image of the $\cO_\p$-module $\BK_\p(K_S)\otimes_{\cO_\p}\!I_{G_S}^r$ inside $\BK_\p(K_S)\otimes_{\cO_\p}\!\cO_\p[G_S]$.

\begin{definition} \label{vanishing-dfn}
Let $r\in\mathbb N$. An element $\theta\in\BK_\p(K_S)\otimes_{\cO_\p}\!\cO_\p[G_S]$ is said to \emph{vanish to order at least $r$} if $\theta\in\BK_\p(K_S)\otimes_{\cO_\p}\!I_{G_S}^r$.
\end{definition}

Analogous definitions and conventions apply to $\cO_\p[\Gamma_S]$ and $I_{\Gamma_S}$ and, below, with $\BK_\p(K_S)^{\otimes2}$ in place of $\BK_\p(K_S)$.

\begin{theorem} \label{main-thm} 
If $\rho_\p\leq p$ then $\theta_{S,\p}\in\BK_\p(K_S)\otimes_{\cO_\p}\!I^{\rho_\p}_{G_S}$.
\end{theorem}

\begin{proof} Let ${\bf D}_\kappa$ be a derivative with $\ord({\bf D}_\kappa)<\rho_\p$, ${\rm supp}({\bf D}_\kappa)=S$ and ${\rm cond}({\bf D}_\kappa)\,|\,S$. Set $S':={\rm cond}({\bf D}_\kappa)$ and write ${\bf D}_\kappa={\bf D}_{\kappa'}\cdot{\bf D}_{\kappa''}$ where the derivative ${\bf D}_{\kappa'}$ satisfies ${\rm supp}({\bf D}_{\kappa'})={\rm cond}({\bf D}_{\kappa'})=S'$ and the derivative ${\bf D}_{\kappa''}$ has order $0$ and support in $S/S'$ (so ${\bf D}_{\kappa''}$ is nothing other than the norm operator from $G_S$ to $G_{S'}$). Part (1) of Proposition \ref{complex-conj} combined with the relation \eqref{res-cores-norm} between Galois trace, restriction and corestriction map shows that 
\begin{equation} \label{eq1}
{\bf D}_{\kappa}(y_{S,\p})=\res_{K_{S'}/K_S}\bigl({\bf D}_{\kappa'}(y_{S',\p})\bigr)\cdot\prod_{\ell\mid (S/S')}a_\ell/\ell^{k/2-1}
\end{equation}
where $\res_{K_{S'}/K_{S}}$ is the restriction from $H^1(K_{S'},A_\p)$ to $H^1(K_{S},A_\p)$. 
Let $p^m=\eta(\kappa)$ denote the highest power of $p$ dividing the orders of the groups $G_\ell$ with $\ell\,|\,S$. By definition, all primes dividing $S$ belong to $\mathcal S_{p^m}$. Since $\rho_\p\leq r_{\p,m} $ by Lemma \ref{lem4.2}, we have $\ord({\bf D}_\kappa)<r_{\p,m}$. Therefore the assumptions of Theorem \ref{div-thm} are satisfied and then 
\begin{equation} \label{eq2}
{\bf D}_{\kappa'}(y_{S',\p})\equiv0\pmod{p^m}.
\end{equation}
Combining \eqref{eq1} and \eqref{eq2}, we see that if $\ord({\bf D}_\kappa)<\rho_\p$ then 
$\eta(\kappa)\,|\,{\bf D}_\kappa(y_S)$. The result follows from the divisibility criterion in \S\ref{3.3.2}, which we can apply thanks to the condition $\rho_\p\leq p$. \end{proof}

\begin{corollary} \label{coro-zeta} 
$\zeta_{S,\p},\zeta^*_{S,\p}\in\BK_\p(K_S)\otimes_{\cO_\p}\!I^{\rho_\p}_{G_S}$.
\end{corollary}

\begin{proof} The element $\zeta_{S,\p}$ is the image of $\theta_{S,\p}$ via the endomorphism of $\Lambda(K_S)\otimes I^{\rho_\p}_{G_S}$ defined by $x\otimes i\mapsto\bigl({\bf N}_S(x)\bigr)\otimes i$. Since the Abel--Jacobi map commutes with Galois actions, it follows from Theorem \ref{main-thm} that $\zeta_{S,\p}$ belongs to $\Lambda(K_S)\otimes I^{\rho_\p}_{G_S}$. Applying the main involution, one obtains that $\zeta_{S,\p}^*$ belongs to $\BK_\p(K_S)\otimes I^{\rho_\p}_{G_S}$ as well. \end{proof}

\begin{corollary} \label{main-coro}
$\mathcal L_{S,\p}\in\BK_\p(K_S)^{\otimes2}\otimes_{\cO_\p}\!I_{\Gamma_S}^{2\rho_\p}$.
\end{corollary}

\begin{proof} Since ${\mathcal L}_{S,\p}$ is a linear combination with coefficients in $\cO_\p[\Gamma_S]$ of the elements $\zeta_{T,\p}$ and $\zeta^*_{T,\p}$ for $T\,|\,S$, the result is a consequence of Corollary \ref{coro-zeta} applied to these elements. \end{proof}

\begin{remark} \label{remark-general-choice}
More generally, the result of Corollary \ref{main-coro} is valid (with the same proof) for any linear combination with coefficients in $\cO_\p[\Gamma_S]$ of the elements $\zeta_{T,\p}$ and $\zeta^*_{T,\p}$ with $T\,|\,S$. See Remark \ref{rem4.1} for a detailed discussion of our specific choice of coefficients for $\mathcal L_{S,\p}$.
\end{remark}

\subsection{Results on the leading terms} \label{sec4.5}

We study, in some particular cases, the reductions modulo $p$ of the leading terms of $\zeta_{S,\p}$ and $\mathcal L_{S,\p}$. Here $S$ is a square-free product of primes in $\mathcal S_p$ and $\rho_\p< p$.

We first consider the leading term $\tilde\theta_{S,\p}$ of $\theta_{S,\p}$, which is defined to be the image of $\theta_{S,\p}$ in $\Lambda_\p(K_S)\otimes_{\cO_\p}\!\bigl(I_{G_S}^{\rho_\p}/I_{G_S}^{\rho_\p+1}\bigr)$. By Theorem \ref{main-thm}, there is a congruence
\begin{equation} \label{eq-theta-45}
\tilde\theta_{S,\p}\equiv\sum_\kappa{\bf D}_\kappa(y_{S,\p})\otimes(\sigma_1-1)^{k_1}\dots(\sigma_t-1)^{k_s}\pmod{p}
\end{equation} 
where the sum is over all the $\kappa$ with $\ord(\kappa)={\rho_\p}$. Denote by 
\[ {\bf D}_\kappa^{(p)}(y_{S,\p})\in\Lambda_\p(K_S)/p\Lambda_\p(K_S) \]
the reduction modulo $p$ of ${\bf D}_\kappa(y_{S,\p})$ for $\ord(\kappa)=\rho_\p$.

\begin{lemma} \label{lemma5.24a} 
${\bf D}_\kappa^{(p)}(y_{S,\p})\in\bigl(\Lambda_\p(K_S)/p\Lambda_\p(K_S)\bigr)^{\!G_S}$.
\end{lemma}

\begin{proof} Combine Theorem \ref{main-thm} and formula \eqref{sigma-ell-D-eq}, for which the condition $\rho_\p<p$ is needed. \end{proof}

Recall that ${\bf N}_S\in\Z[\Gamma_S]$ is a lift of the norm operator in $\Z[\Gamma_1]$.
 
\begin{lemma} \label{lemma5.24} 
${\bf D}_\kappa^{(p)}\bigl({\bf N}_S(y_{S,\p})\bigr)\in\bigl(\Lambda_\p(K_S)/p\Lambda_\p(K_S)\bigr)^{\!\Gamma_S}$.
\end{lemma}

\begin{proof} Immediate from Lemma \ref{lemma5.24a}. \end{proof}

By \eqref{eq13}, there is an injection $\Lambda_\p(K)/p\Lambda_\p(K)\hookrightarrow H^1_f(K,W_\p[p])$. Define the $\p$-part $\Sha_\p(f/K)$ of the Shafarevich--Tate group of $f$ over $K$ as the cokernel of this map, so that there is a short exact sequence 
\[ 0\longrightarrow \Lambda_\p(K)/p\Lambda_\p(K)\longrightarrow H^1_f(K,W_\p[p])\longrightarrow\Sha_\p(f/K)\longrightarrow0. \] 
See \cite[p. 118]{Nek} for a slightly different definition of $\Sha_\p(f/K)$. 

\begin{proposition} \label{prop5.24} 
Suppose that $|\rho_\p^+-\rho_\p^-|=1$ and let $\bf D_\kappa$ have order $\rho_\p$ and support $S$. If ${\bf D}_\kappa(y_{S,\p})\not\equiv0\pmod p$ then
\begin{enumerate}
\item $\Sha_\p(f/K)=0$;
\item with $A(S)$ defined as in \eqref{A(S')-eq}, the natural map
\[ H^1_f(K,W_\p[p])\longrightarrow A(S) \] 
is surjective. 
\end{enumerate}
\end{proposition}

\begin{proof} We first observe that, by definition, one has 
\begin{equation} \label{eq41}
2\rho_\p=\dim_{F_\p}\bigl(X_\p(K)\bigr)-1.
\end{equation} 
Proposition \ref{T5.20} shows that 
\begin{equation} \label{eq41-1}
\rho_\p\geq r_{\p}\bigl(H^1_{f,S}(K,W_\p[p^m])^{\epsilon_\kappa}\bigr)+r_\p\bigl(A(S)^{\epsilon_\kappa}\bigr)\geq r_\p\bigl(H^1_f(K,W_\p[p^m])^{\epsilon_\kappa}\bigr),
\end{equation} 
while Corollary \ref{coro5.21} implies that 
\begin{equation} \label{eq41-2}
\rho_\p\geq r_\p\bigl(H^1_{f,S}(K,W_\p[p^m])^{-\epsilon_\kappa}\bigr)+r_\p\bigl(A(S)^{-\epsilon_\kappa}\bigr)-1\geq r_\p\bigl(H^1_f(K,W_\p[p^m])^{-\epsilon_\kappa}\bigr)-1
\end{equation}
(for the last inequalities in the above chains, see the proof of Corollary \ref{coro5.22}). Since 
\begin{equation}\label{eq42}
r_\p\bigl(H^1_f(K,W_\p[p^m])^\pm\bigr)\geq r_\p\Big(\bigl(\Lambda_\p(K)/p\Lambda_\p(K)\bigr)^\pm\Big)\geq\rho_\p^\pm,
\end{equation}
we obtain the inequalities $\rho_\p\geq\rho_\p^{\epsilon_\kappa}$ and $\rho_\p\geq\rho_\p^{-\epsilon_\kappa}-1$. Therefore we have the inequality 
\begin{equation} \label{eq43}
2\rho_\p\geq\dim_{F_\p}\bigl(X_\p(K)\bigr)-1.
\end{equation} 
Comparing \eqref{eq41} and \eqref{eq43}, we conclude that all the above inequalities are, in fact, equalities; in particular, the first inequality in \eqref{eq42} is an equality, from which (1) follows immediately by definition of $\Sha_\p(f/K)$. Furthermore, the second inequalities in \eqref{eq41-1} and \eqref{eq41-2} are equalities, and then 
\[ r_\p\bigl(H^1_f(K,W_\p[p^m])\bigr)=r_\p\bigl(H^1_{f,S}(K,W_\p[p^m])\bigr)+r_\p\bigl(A(S)\bigr). \] 
Comparing this equality with the definition of $H^1_{f,S}(K,W_\p[p^m])$ in \eqref{S'-Selmer} proves (2). \end{proof}

\begin{proposition} \label{prop5.26}
If $|\rho_\p^+-\rho_\p^-|=1$ and $\ord(\kappa)={\rho_\p}$ then ${\bf D}_\kappa^{(p)}\bigl({\bf N}_S(y_{S,\p})\bigr)$ lies in the image of $\Lambda_\p(K)/p\Lambda_\p(K)$.
\end{proposition}

\begin{proof} By \eqref{eq13}, there is an injective map
\begin{equation} \label{injection-Lambda-quotient-eq} 
\Lambda_\p(K_S)/p\Lambda_\p(K_S)\;\longmono\;H^1_f(K_S,W_\p[p])\subset H^1(K_S,W_\p[p]). 
\end{equation}
Recall that restriction gives an isomorphism $H^1(K,W_\p[p])\simeq H^1(K_S,W_\p[p])^{\Gamma_S}$ and that ${\bf D}_\kappa^{(p)}({\bf N}_Sy_{S,\p})$ belongs to $(\Lambda_\p(K_S)/p\Lambda_\p(K_S))^{\Gamma_S}$ by Lemma \ref{lemma5.24}. In light of these facts and the $\Gamma_S$-equivariance of the injection \eqref{injection-Lambda-quotient-eq}, write $d$ for the image of ${\bf D}_\kappa^{(p)}({\bf N}_Sy_{S,\p})$ in $H^1(K,W_\p[p])$. 

We first show that $d\in H^1_f(K,W_\p[p])$. By an argument similar to those in Propositions \ref{prop5.9} and \ref{finite-1-prop}, one can check that the restriction of $d$ at all places $v\nmid S$ is finite. There is a map  
\begin{equation} \label{eq46}
\bigoplus_{v|S}H^1_{sin}(K_v,W_\p[p])\longrightarrow {H^1_f(K,W_\p[p])}^*
\end{equation}
taking $x={(x_v)}_{v|S}$ to the linear function 
\[ s\longmapsto\sum_{v\in S}\langle x,\res_v(s)\rangle_v \]
on $H^1_f(K,W_\p[p])$ (recall that all the primes dividing $S$ are inert in $K$). Since $d$ is a global class, Tate duality ensures that the image of $d$ in $\oplus_{v|S}H^1_{sin}(K_v,W_\p[p])$ 
belongs to the kernel of \eqref{eq46}. With $A(S)$ as in \eqref{A(S')-eq}, part (2) of Proposition \ref{prop5.24} shows that the map 
\[ H^1_f(K,W_\p[p])\longrightarrow A(S) \] 
is surjective and hence, dually, that the map in \eqref{eq46} is injective (here we are implicitly using isomorphism \eqref{isom-pontryagin-eq}). It follows that $d$ is locally finite everywhere and belongs to $H^1_f(K,W_\p[p])$. 

Since $\Sha_\p(f/K)=0$ by part (1) of Proposition \ref{prop5.24}, we conclude that $d$ comes from a class in $\Lambda_\p(K)/\Lambda_\p(K)$. But there is a commutative diagram  
\[ \xymatrix{\Lambda_\p(K)/p\Lambda_\p(K)\ar[d]\ar[r]^-\simeq&H^1_f(K,W_\p[p])\ar[d]\ar@{^(->}[r]&H^1(K,W_\p[p])\ar[d]^-\simeq\\
\bigl(\Lambda_\p(K_S)/p\Lambda_\p(K_S)\bigr)^{\Gamma_S}\ar@{^(->}[r]&H^1_f(K_S,W_\p[p]{)}^{\Gamma_S}\ar@{^(->}[r]&H^1(K_S,W_\p[p]{)}^{\Gamma_S}} \]
in which all the horizontal arrows are injective, and the proposition follows. \end{proof}

The information collected above on ${\bf D}_\kappa^{(p)}\bigl({\bf N}_S(y_{S,\p})\bigr)$ when $\ord(\kappa)={\rho_\p}$ yields a result on the reduction modulo $p$ of the leading term $\tilde\zeta_{S,\p}$ of $\zeta_{S,\p}$. More precisely, define $\tilde\zeta_{S,\p}$ as the image of $\zeta_{S,\p}$ in $\Lambda_\p(K_S)\otimes_{\cO_\p}\!\bigl(I_{\Gamma_S}^{\rho_\p}/I_{\Gamma_S}^{\rho_\p+1}\bigr)$ and consider its mod $p$ reduction
\[ \tilde\zeta_{S,\p}^{(p)}\in\bigl(\Lambda_\p(K_S)/p\Lambda_\p(K_S)\bigr)\otimes_{\cO_\p}\!\bigl(I_{\Gamma_S}^{\rho_\p}/I_{\Gamma_S}^{\rho_\p+1}\bigr). \]
Finally, 
let $J(S)$ denote the cokernel of the map $H^1_f(K,W_\p[p])\rightarrow A(S)$; see \eqref{A(S')-eq} with $S'=S$ and $m=1$ for the definition of $A(S)$. 

\begin{theorem} \label{theorem5.27} 
Fix a square-free product $S$ of primes in $\mathcal S_p$.
\begin{enumerate}
\item $\tilde{\zeta}_{S,\p}^{(p)}\in\bigl(\Lambda_\p(K_S)/p\Lambda_\p(K_S)\bigr)^{\Gamma_S}\otimes_{\cO_\p}\!\bigl(I^{\rho_\p}_{G_S}/I^{\rho_\p+1}_{S,\p}\bigr)$.
\item If $|\rho_\p^+-\rho_\p^-|=1$ then $\tilde\zeta_{S,\p}^{(p)}$ belongs to the image of the map 
\[ \bigl(\Lambda_\p(K)/p\Lambda_\p(K)\bigr)\otimes_{\mathcal O_\p}\!\bigl(I_{G_S}^{\rho_\p}/I_{G_S}^{{\rho_\p}+1}\bigr)\longrightarrow\bigl(\Lambda_\p(K_S)/p\Lambda_\p(K_S)\bigr)^{\Gamma_S}\otimes_{\cO_\p}\!\bigl(I_{G_S}^{\rho_\p}/I_{G_S}^{{\rho_\p}+1}\bigr). \]
\item Assume that $\Sha_\p(f/K)$ is finite. If $|\rho_\p^+-\rho_\p^-|=1$ and $p$ divides $|\Sha_\p(f/K)|\cdot|J(S)|$ then $\tilde\zeta_{S,\p}^{(p)}=0$.  
\end{enumerate}
\end{theorem}

\begin{proof} Part (1) follows from \eqref{eq-theta-45} and Lemma \ref{lemma5.24}, while part (2) follows from \eqref{eq-theta-45} and Proposition \ref{prop5.26}. As for part (3), if $\tilde\zeta_{S,\p}^{(p)}\not=0$ then \emph{a fortiori} ${\bf D}_\kappa^{(p)}({\bf N}_Sy_{S,\p})\not\equiv 0$ for all $\kappa$ with $\ord(\kappa)={\rho_\p}$, and so Proposition \ref{prop5.24} gives the triviality of both $\Sha_\p(f/K)$ and $J(S)$. \end{proof}
 
\begin{corollary} \label{coro4.14} 
Fix a square-free product $S$ of primes in $\mathcal S_p$.  
\begin{enumerate}
\item The image $\tilde{\mathcal L}_{S,\p}^{(p)}$ of $\mathcal L_{S,\p}$ in 
\[ \bigl(\Lambda_\p(K_S)^{\otimes 2}/p\Lambda_\p(K_S)^{\otimes 2}\bigr)\otimes_{\cO_\p}\!\bigl(I_{\Gamma_S}^{2\rho_\p}/I_{\Gamma_S}^{2\rho_\p+1}\bigr) \]
belongs to the image of 
\[ \bigl(\Lambda_\p(K_S)^{\otimes2}/p\Lambda_\p(K_S)^{\otimes2}\bigr)^{\Gamma_S}\otimes_{\cO_\p}\!\bigl(I_{\Gamma_S}^{2\rho_\p}/I_{\Gamma_S}^{2\rho_\p+1}\bigr). \]
\item If $|\rho_\p^+-\rho_\p^-|=1$ then $\tilde{\mathcal L}_{S,\p}^{(p)}$ belongs to the image of
\[ \bigl(\Lambda_\p(K)^{\otimes2}/p\Lambda_\p(K)^{\otimes2}\bigr)\otimes_{\cO_\p}\!\bigl(I_{\Gamma_S}^{2{\rho_\p}}/I_{\Gamma_S}^{2{\rho_\p}+1}\bigr). \]
\item Assume that $\Sha_\p(f/K)$ is finite. If $|\rho_\p^+-\rho_\p^-|=1$ and $p$ divides $|\Sha_\p(f/K)|\cdot|J(S)|$ then $\tilde{\mathcal L}_{S,\p}^{(p)}=0$.  
\end{enumerate}
\end{corollary}

\begin{proof}  The term $\mathcal L_{S,\p}$ is an $\cO_\p[\Gamma_S]$-linear combination of the elements $\zeta_{T,\p}$ and $\zeta^*_{T,\p}$ for $T\,|\,S$, and the result is obtained by applying Theorem \ref{theorem5.27} to each of them. \end{proof}

\begin{remark} \label{remark-general-choice-2} 
In parallel with Remark \ref{remark-general-choice}, we observe that the results of Corollary \ref{coro4.14} hold more generally for any $\cO_\p[\Gamma_S]$-linear combination of the elements $\zeta_{T,\p}$ and $\zeta^*_{T,\p}$ for $T\,|\,S$. 
\end{remark}

\subsection{Galois module structure of Heegner cycles} 

Fix a prime number $\ell\in\mathcal S_p$. Define $\mathcal H(K_\ell)$ to be the $\mathcal O_\p[G_\ell]$-module generated by $y_{\ell,\p}$ inside $\BK_\p(K_\ell)$ and denote by $\mathcal H_p(K_\ell)$ the $\mathbb F_\p$-subspace $\mathcal H(K_\ell)/p\mathcal H(K_\ell)$ of $\BK_\p(K_\ell)/p\BK_\p(K_\ell)$. Finally, recall from \S \ref{main-subsec} that $r_\p=r_{\p,1}$. 

\begin{theorem} \label{Heegner-theorem}
$\dim_{\mathbb F_\p}\!\bigl(\mathcal H_p(K_\ell)\bigr)\leq\ell+1-r_\p$. 
\end{theorem}

\begin{proof} By \S\ref{3.2.5}, an $\cO_\p$-basis of $\mathcal H(K_\ell)$ is given by $\bigl\{{\bf D}_\ell^i(y_{\ell,\p})\mid i=0,\dots,\ell\bigr\}$. Theorem \ref{div-thm} shows then that ${\bf D}_{\ell}^k(y_{\ell,\p})\equiv0\pmod{p}$ if $k<r_\p$, and hence at most $\ell+1-r_\p$ elements of the $\mathcal O_\p$-basis of $\mathcal H(K_\ell)$ under consideration are non-zero. \end{proof}

\section{Regulators} \label{sec-regulators}

In this final section we propose an axiomatic construction of regulators in our context, generalizing those introduced by Mazur and Tate in \cite{MT0} and \cite{MT} and used in Darmon's work \cite{Dar}. We believe that this picture can be made unconditional by adopting the point of view of Nekov\'a\v{r} (\cite{nek4}). We also observe that a cohomological approach to Mazur--Tate regulators, along the lines of the theory developed by Bertolini and Darmon in \cite{BD1} and \cite{BD2}, is certainly possible, at least in some special situations: we plan to come back to these issues in a future project. 

Let $K$ be an imaginary quadratic field of discriminant coprime to $Np$ and let $S>1$ be a square-free product of primes that are inert in $K$, then define 
\begin{equation} \label{eq52}
\Lambda_{\p,S}(K):=\ker\Bigg(\!\Lambda_\p(K)\longrightarrow\bigoplus_{\lambda\mid S}H^1_f(K_\lambda,A_\p)\!\Bigg)
\end{equation} 
where the map is induced by \eqref{lambda-dfn-eq} via localizations. Finally, recall the maps $\mu_{S,T}$ introduced in \S\ref{sec-theta}, which are defined for integers $T\,|\,S$. Our starting point is the following

\begin{assumption} \label{ass1}
There exists a bilinear pairing  
\[ {\langle\cdot\,,\cdot\rangle}_S:\BK_\p(K)\times\BK_{\p,S}(K)\longrightarrow I_{\Gamma_S}/I_{\Gamma_S}^2\]
satisfying the compatibility condition 
\[ \mu_{S,T}\circ{\langle\cdot\,,\cdot\rangle}_S={\langle\cdot\,,\cdot\rangle}_T \] 
for all $T\,|\,S$.
\end{assumption}

As in \cite{Dar}, we use this pairing to construct a regulator map. Let $\ell$ be a prime number dividing $S$ and, as before, let $\lambda$ be the unique prime of $K$ above $\ell$, then write $F_\lambda$ for the arithmetic Frobenius in $\Gal(\Q_\ell^{\rm nr}/K_\lambda)$. Recall from \S \ref{selmer-subsec} that $V_\p=A_\p\otimes_{\cO_\p}F_\p$. 
We have 
\[ \det(F_\ell\pm1\,|\,V_\p)=\ell+1\mp\frac{a_\ell}{\ell^{\frac{k}{2}-1}}, \] 
which are non-zero thanks to the Weil bounds, hence 
\[ \begin{split}
   \det(F_\lambda-1\,|\,V_\p)&=\det(F_\ell+1\,|\,V_\p)\cdot\det(F_\ell-1\,|\,V_\p)\\&=(\ell+1)^2-\frac{a_\ell^2}{\ell^{k-2}}\not=0. 
    \end{split} \]
Then \cite[Theorem 4.1, (i)]{BK} implies that $H^1_f(K_\lambda,A_\p)$ is finite, so the codomain of the map in \eqref{eq52} is finite and the ranks of $\Lambda_{\p,S}(K)$ and $\Lambda_\p(K)$ over $\cO_\p$ are equal. As in \eqref{tilde-rho-eq}, this common rank will be denoted by $\tilde \rho_\p$. Fix finite index subgroups $A\subset\Lambda_\p(K)$ and $B\subset\Lambda_{\p,S}(K)$ that are $\mathcal O_\p$-free and choose $\mathcal O_\p$-bases $\{P_1,\dots,P_{\tilde \rho_\p}\}$ and $\{Q_1,\dots,Q_{\tilde \rho_\p}\}$ of $A$ and $B$, respectively. Form the matrix 
\[ R(A,B):=\bigl({\langle P_i,Q_j\rangle}_S\bigr)_{i,j=1,\dots,\tilde\rho_\p} \] 
with entries in $I_{\Gamma_S}/I_{\Gamma_S}^2$ and let $R_{i,j}(A,B)$ be the $(i,j)$-minor of $R(A,B)$. Consider the element 
\[ \Reg(A,B):=\sum_{i,j=1}^{\tilde \rho_\p}(-1)^{i+j}(P_i\otimes Q_j)\otimes \det\bigl(R_{i,j}(A,B)\bigr)\in\Lambda_\p(K)^{\otimes 2}\otimes\bigl(I_{\Gamma_S}^{{\tilde\rho_\p}-1}/I_{\Gamma_S}^{\tilde\rho_\p}\bigr) \] 
and define the \emph{regulator term} $\Reg(S)$ as 
\[{\rm Reg}(S):={\Reg}(A,B)/([\BK_\p(K):A]\cdot[\BK_{\p,S}(K):B]).\]
This is independent of the choice of $A$ and $B$.

Let $B(S)$ denote the cokernel of the map in \eqref{eq52}, so that there is an exact sequence
\[ 0\longrightarrow\Lambda_{\p,S}(K)\longrightarrow\Lambda_\p(K)\longrightarrow\bigoplus_{\lambda\mid S}H^1_f(K_\lambda,A_\p)\longrightarrow B(S)\longrightarrow0. \] 
The analogue of Darmon's conjecture \cite[Conjecture 2.3]{Dar} in the present context amounts to the following 

\begin{question} \label{ques}
Suppose that $\Sha_\p(f/K)$ is finite. Is there a pairing as in Assumption \ref{ass1} such that the following conditions are satisfied? 
\begin{enumerate}
\item $\mathcal L_{S,\p}\in \BK(K_S)^{\otimes 2}\otimes I_{\Gamma_S}^{\tilde \rho_\p-1}$.
\item The leading coefficient $\tilde{\mathcal L}_{S,\p}$ of $\mathcal L_{S,\p}$ in $\BK(K_S)^{\otimes2}\otimes\bigl(I_{\Gamma_S}^{\tilde\rho_\p-1}/I_{\Gamma_S}^{\tilde \rho_\p}\bigr)$ belongs to the image of the canonical map 
\begin{equation} \label{canonical-augmentation-eq} 
\BK(K)^{\otimes2}\otimes\bigl(I_{\Gamma_S}^{\tilde \rho_\p-1}/I_{\Gamma_S}^{\tilde \rho_\p}\bigr)\longrightarrow\BK(K_S)^{\otimes2}\otimes\bigl(I_{\Gamma_S}^{\tilde \rho_\p-1}/I_{\Gamma_S}^{\tilde \rho_\p}\bigr). 
\end{equation}
\item There exists $c(f)\in\cO_\p$ depending only on $f$ such that 
\[ \tilde{\mathcal L}_{S,\p}=c(f)\cdot |\Sha_\p(f/K)|\cdot |B(S)|\cdot {\rm Reg}(S). \] 
Here ${\rm Reg}(S)$ denotes also the image of ${\rm Reg}(S)$ via the map in \eqref{canonical-augmentation-eq}, by an abuse of notation. 
\end{enumerate}
\end{question}

A few remarks are in order here. First of all, parts (1) and (2) of Question \ref{ques} do not depend 
on ${\rm Reg}(S)$, so they can be formulated independently of Assumption \ref{ass1}. Moreover, 
thanks to the compatibility condition in Assumption \ref{ass1}, one can show that
\[ \mu_{S,T}\bigl(|B(T)|\cdot \Reg(T)\bigr)=|B(S)|\cdot \Reg(S)\cdot\prod_{\ell\mid(S/T)}(1+\ell-a_\ell/\ell^{k/2-1})\cdot(1+\ell+a_\ell/\ell^{k/2-1}) \] 
whenever $T\,|\,S$. Comparing with \eqref{compatibility}, one sees that (1), (2) and (3) above are all compatible with the map $\mu_{S,T}$ when $T\,|\,S$. Actually, as in \cite{Dar}, it is this compatibility relation that suggests the definition of $\mathcal L_{S,\p}$ given above. However, 
different regulators might be attached to different choices of the coefficients $b_T$ and $b_T'$, 
as discussed in Remark \ref{rem4.1}. The choice of \emph{correct} regulators and $\mathcal L$-elements is an open problem, although we believe that, in light of \eqref{compatibility} and the fact that the relation in Assumption \ref{ass1} is very natural, our definition of $\mathcal L_{S,\p}$ is 
in some sense the ``standard'' one. Furthermore, since $2\rho_p\geq\tilde\rho_\p-1$ with equality holding if and only if $|\rho_\p^+-\rho_\p^-|=1$, we can state the following partial results. 
\begin{enumerate}
\item Condition (1) in Question \ref{ques} is implied by Corollary \ref{main-coro}, and is equivalent to it if $|\rho_\p^+-\rho_\p^-|=1$. 
\item If $|\rho_\p^+-\rho_\p^-|=1$ and all the prime factors of $S$ belong to $\mathcal S_p$ then a mod $p$ version of (2) in Question \ref{ques} is given in part (2) of Corollary \ref{coro4.14}. 
\item Suppose that $|\rho_\p^+-\rho_\p^-|=1$ and all the prime factors of $S$ belong to $\mathcal S_p$. In this case, if $p$ divides $|\Sha_\p(f/K)|\cdot |B(S)|$ then $p$ divides $\tilde{\mathcal L}_{S,\p}$ as well. This is a simple consequence of part (3) of Corollary \ref{coro4.14}.
\end{enumerate} 
Let us finally observe that if $|\rho_\p^+-\rho_\p^-|>1$ then (1) predicts more than what is proved in Corollary \ref{main-coro}, and the question about the leading term of $\mathcal L_{S,\p}$ might be addressed, at least in some special cases, by means of a suitable theory of generalized regulators as in \cite{BD1} and \cite{BD2}.


\begin{thebibliography}{100}

\bibitem{Beil} A. A. Beilinson, Higher regulators and values of $L$-functions, \emph{J. Math. Sci. (N. Y.)} {\bf 30} (1985), no. 2, 2036--2070.

\bibitem{BD1} M. Bertolini, H. Darmon, Derived heights and generalized Mazur--Tate regulators, \emph{Duke Math. J.} {\bf 76} (1994), no. 1, 75--111. 

\bibitem{BD2} M. Bertolini, H. Darmon, Derived $p$-adic heights, \emph{Amer. J. Math.} {\bf 117} (1995), no. 6, 1517--1554.

\bibitem{BD96} M. Bertolini, H. Darmon, Heegner points on Mumford--Tate curves, \emph{Invent. Math.} {\bf 126} (1996), no. 3, 413--456.

\bibitem{BD98} M. Bertolini, H. Darmon, Heegner points, $p$-adic $L$-functions, and the Cerednik--Drinfeld uniformization, \emph{Invent. Math.} {\bf 131} (1998), no. 3, 453--491.

\bibitem{BD99} M. Bertolini, H. Darmon, $p$-adic periods, $p$-adic $L$-functions and the $p$-adic uniformization of Shimura curves, \emph{Duke Math. J.} {\bf 98} (1999), no. 2, 305--334.

\bibitem{BD01} M. Bertolini, H. Darmon, The $p$-adic $L$-functions of modular elliptic curves, in \emph{Mathematics unlimited -- 2001 and beyond}, B. Engquist and W. Schmid (eds.), Springer-Verlag, Berlin, 2001, 109--170.

\bibitem{BD-Iwasawa} M. Bertolini, H. Darmon, Iwasawa's Main Conjecture for elliptic curves over anticyclotomic $\Z_p$-extensions, \emph{Ann. of Math. (2)} {\bf 162} (2005), no. 1, 1--64.

\bibitem{BDP} M. Bertolini, H. Darmon, K. Prasanna, Generalised Heegner cycles and $p$-adic Rankin $L$-series, \emph{Duke Math. J.}, to appear.


\bibitem{Bes} A. Besser, On the finiteness of $\Shaa$ for motives associated to modular forms, \emph{Doc. Math.} {\bf 2} (1997), 31--46.

\bibitem{Bloch} S. Bloch, Algebraic cycles and values of $L$-functions, \emph{J. Reine Angew. Math.} {\bf 350} (1984), 94--108.

\bibitem{BK} S. Bloch, K. Kato, $L$-functions and Tamagawa numbers of motives, in \emph{The Grothendieck Festschrift, vol. I}, P. Cartier, L. Illusie, N. M. Katz, G. Laumon, Y. Manin and K. A. Ribet (eds.), Progress in Mathematics {\bf 86}, Birkh\"auser, Boston, MA, 1990, 333--400.

\bibitem{BFH} D. Bump, S. Friedberg, J. Hoffstein, Nonvanishing theorems for $L$-functions of modular forms and their derivatives, \emph{Invent. Math.} {\bf 102} (1990), no. 3, 543--618.

\bibitem{Bu}  D. Burns, Congruences between derivatives of abelian $L$-functions at $s=0$, \emph{Invent. Math.} {\bf 169} (2007), no. 3, 451--499.

\bibitem{BF} D. Burns, M. Flach, Tamagawa numbers for motives with (non-commutative) coefficients, \emph{Doc. Math.} {\bf 6} (2001), 501--570.

\bibitem{Cas} F. Castella, Heegner cycles and higher weight specializations of big Heegner points, \emph{Math. Ann.}, to appear.

\bibitem{CH} M. Chida, M.-L. Hsieh, Anticyclotomic Iwasawa main conjecture for modular forms, in progress. 

\bibitem{Dar} H. Darmon, A refined conjecture of Mazur--Tate type for Heegner points, \emph{Invent. Math.} {\bf 110} (1992), no. 1, 123--146. 


\bibitem{Del} P. Deligne, Formes modulaires et repr\'esentations $\ell$-adiques, Lecture Notes in Mathematics {\bf 179}, Springer-Verlag, Berlin, 1971, 139--172.


\bibitem{DFG} F. Diamond, M. Flach, L. Guo, The Tamagawa number conjecture of adjoint motives of modular forms, \emph{Ann. Sci. \'Ec. Norm. Sup. (4)} {\bf 37} (2004), no. 5, 663--727.

\bibitem{Fou} O. Fouquet, Dihedral Iwasawa theory of nearly ordinary quaternionic automorphic forms, \emph{Compos. Math.}, to appear.

\bibitem{Gr} B. H. Gross, On the values of abelian $L$-functions at $s=0$, \emph{J. Fac. Sci. Univ. Tokyo Sect. IA Math.} {\bf 35} (1988), no. 1, 177--197.


\bibitem{Ho} B. Howard, Variation of Heegner points in Hida families, \emph{Invent. Math.} {\bf 167} (2007), no. 1, 91--128.



\bibitem{Jan} U. Jannsen, \emph{Mixed motives and algebraic $K$-theory}, Lecture Notes in Mathematics {\bf 1400}, Springer-Verlag, Berlin, 1990.

\bibitem{Lo1} M. Longo, On the Birch and Swinnerton-Dyer conjecture for modular elliptic curves over totally real fields, \emph{Ann. Inst. Fourier (Grenoble)} {\bf 56} (2006), no. 3, 689--733.

\bibitem{Lo2} M. Longo, Euler systems obtained from congruences between Hilbert modular forms, \emph{Rend. Semin. Mat. Univ. Padova} {\bf 118} (2007), 1--34.

\bibitem{Lo3} M. Longo, Anticyclotomic Iwasawa's main conjecture for Hilbert modular forms, \emph{Comment. Math. Helv.} {\bf 87} (2012), no. 2, 303--353.

\bibitem{LRV} M. Longo, V. Rotger, S. Vigni, Special values of $L$-functions and the arithmetic of Darmon points, \emph{J. Reine Angew. Math.}, to appear. 

\bibitem{LV-JNT} M. Longo, S. Vigni, On the vanishing of Selmer groups for elliptic curves over ring class fields, \emph{J. Number Theory} {\bf 150} (2010), no. 1, 128--163.

\bibitem{LV-Man} M. Longo, S. Vigni,  Quaternion algebras, Heegner points and the arithmetic of Hida families, \emph{Manuscripta Math.} {\bf 135} (2011), no. 3--4, 273--328.

\bibitem{LV} M. Longo, S. Vigni, Vanishing of special values and central derivatives in Hida families, \emph{Ann. Sc. Norm. Super. Pisa Cl. Sci.}, to appear.

\bibitem{MT0} B. Mazur, J. Tate, Canonical height pairings via biextensions, in \emph{Arithmetic and Geometry, vol. I}, M. Artin and J. Tate (eds.), Progress in Mathematics {\bf 35}, Birkh\"auser, Boston, MA, 1983, 195--237.

\bibitem{MT} B. Mazur, J. Tate, Refined conjectures of the ``Birch and Swinnerton-Dyer type'', \emph{Duke Math. J.} {\bf 54} (1987), no. 2, 711--750.

\bibitem{Milne} J. S. Milne, \emph{Arithmetic duality theorems}, Perspectives in Mathematics {\bf 1}, Academic Press, Boston, MA, 1986. 

\bibitem{Nek} J. Nekov\'a\v{r}, Kolyvagin's method for Chow groups of Kuga--Sato varieties, \emph{Invent. Math.} {\bf 107} (1992), no. 1, 99--125.

\bibitem{nek4} J. Nekov\'a\v{r}, On $p$-adic height pairings, in \emph{S\'eminaire de Th\'eorie des Nombres, Paris, 1990--91}, S. David (ed.), Progress in Mathematics {\bf 108}, Birkh\"auser, Boston, MA, 1993, 127--202.  

\bibitem{Nek2} J. Nekov\'a\v{r}, On the $p$-adic height of Heegner cycles, \emph{Math. Ann.} {\bf 302} (1995), no. 1, 609--686. 

\bibitem{Nek3} J. Nekov\'a\v{r}, $p$-adic Abel--Jacobi maps and $p$-adic heights, in \emph{The arithmetic and geometry of algebraic cycles (Banff, AB, 1998)}, B. B. Gordon, J. D. Lewis, S. M\"uller-Stach, S. Saito and N. Yui (eds.), CRM Proceedings \& Lecture Notes {\bf 24}, American Mathematical Society, Providence, RI, 2000, 367--379. 

\bibitem{Nek-Dur} J. Nekov\'a\v{r}, The Euler system method for CM points on Shimura curves, in \emph{$L$-functions and Galois representations}, D. Burns, K. Buzzard and J. Nekov\'a\v{r} (eds.), London Mathematical Society Lecture Note Series {\bf 320}, Cambridge University Press, Cambridge, 2007, 471--547. 

\bibitem{Nek-Level} J. Nekov\'a\v{r}, Level raising and Selmer groups for Hilbert modular forms of weight two, \emph{Canad. J. Math.} {\bf 64} (2012), no. 3, 588--668. 

\bibitem{Niz} W. Nizio\l, On the image of $p$-adic regulators, \emph{Invent. Math.} {\bf 127} (1997), no. 2, 375--400.

\bibitem{Rib} K. A. Ribet, On $l$-adic representations attached to modular forms II, \emph{Glasgow Math. J.} {\bf 27} (1985), 185--194.

\bibitem{Ru} K. Rubin, A Stark conjecture ``over $\Z$'' for abelian $L$-functions with multiple zeros, \emph{Ann. Inst. Fourier (Grenoble)} {\bf 46} (1996), no. 1, 33--62.

\bibitem{Sai} T. Saito, Modular forms and $p$-adic Hodge theory, \emph{Invent. Math.} {\bf 129} (1997), no. 3, 607--620.

\bibitem{Sai2} T. Saito, Weight-monodromy conjecture for $l$-adic representations associated to modular forms, in \emph{The arithmetic and geometry of algebraic cycles (Banff, AB, 1998)}, B. B. Gordon, J. D. Lewis, S. M\"uller-Stach, S. Saito and N. Yui (eds.), NATO Science Series C: Mathematical and Physical Sciences {\bf 548}, Kluwer Academic Publishers, Dordrecht, 2000, 427--431.

\bibitem{Schneider} P. Schneider, Introduction to the Beilinson conjectures, in \emph{Beilinson's conjectures on special values of $L$-functions}, M. Rapoport, N. Schappacher and P. Schneider (eds.), Perspectives in Mathematics {\bf 4}, Academic Press, Boston, MA, 1988, 1--35.

\bibitem{Sch} A. J. Scholl, Motives for modular forms, \emph{Invent. Math.} {\bf 100} (1990), no. 2, 419--430. 

\bibitem{Se} J.-P. Serre, \emph{Local fields}, Graduate Texts in Mathematics {\bf 67}, Springer-Verlag, New York, 1979.

\bibitem{Zh} S. Zhang, Heights of Heegner cycles and derivatives of $L$-series, \emph{Invent. Math.} {\bf 130} (1997), no. 1, 99--152.

\end{thebibliography}
\end{document}